\documentclass[10pt,a4paper,twoside,reqno]{amsart}

\usepackage{xcolor,soul,lipsum}
\usepackage{hyperref,url}
\usepackage{amssymb,latexsym}
\usepackage{amsmath,amsthm}
\usepackage[latin1]{inputenc}
\usepackage{enumerate}
\usepackage{pict2e}
\usepackage[all]{xy} \CompileMatrices
\usepackage{amscd}
\usepackage{comment}
\hypersetup{colorlinks=false,pdfborderstyle={/S/U/W 0}}
\usepackage{slashed} 
\usepackage{tikz}
\usetikzlibrary{matrix}
\usepackage{multirow}

\SelectTips{cm}{12} 
\theoremstyle{plain}
\newtheorem{theorem}{Theorem}[section]
\newtheorem{lemma}[theorem]{Lemma}
\newtheorem{corollary}[theorem]{Corollary}
\newtheorem{proposition}[theorem]{Proposition}

\theoremstyle{definition}
\newtheorem{definition}[theorem]{Definition}
\newtheorem{definition-theorem}[theorem]{Definition-Theorem}
\newtheorem{example}[theorem]{Example}
\theoremstyle{remark}
\newtheorem{remark}[theorem]{Remark}

\numberwithin{equation}{section} 

\usepackage{geometry}
\usepackage{enumitem}

\newgeometry{left=3cm,right=3cm,top=2.5cm,bottom=2.5cm}

\newcommand{\tll}{\Theta_{ll}}
\newcommand{\tnn}{\Theta_{nn}}
\newcommand{\tln}{\Theta_{ln}}
\newcommand{\tun}{\Theta_{un}}
\newcommand{\tul}{\Theta_{ul}}
\newcommand{\tuu}{\Theta_{uu}}
\newcommand{\Bt}{\mathcal{B}_t}
\makeatletter
\DeclareRobustCommand{\loplus}{\mathbin{\mathpalette\dog@lsemi{+}}}
\DeclareRobustCommand{\lotimes}{\mathbin{\mathpalette\dog@lsemi{\times}}}
\DeclareRobustCommand{\roplus}{\mathbin{\mathpalette\dog@rsemi{+}}}
\DeclareRobustCommand{\rotimes}{\mathbin{\mathpalette\dog@rsemi{\times}}}

\newcommand{\dog@rsemi}[2]{\dog@semi{#1}{#2}{-90,90}}
\newcommand{\dog@lsemi}[2]{\dog@semi{#1}{#2}{270,90}}
\newcommand{\dog@semi}[3]{%
  \begingroup
  \sbox\z@{$\m@th#1#2$}%
  \setlength{\unitlength}{\dimexpr\ht\z@+\dp\z@\relax}%
  \makebox[\wd\z@]{\raisebox{-\dp\z@}{%
    \begin{picture}(1,1)
    \linethickness{\variable@rule{#1}}
    \roundcap
    \put(0.5,0.5){\makebox(0,0){\raisebox{\dp\z@}{$\m@th#1#2$}}}
    \put(0.5,0.5){\arc[#3]{0.5}}
    \end{picture}%
  }}%
  \endgroup
}
\newcommand{\variable@rule}[1]{%
  \fontdimen8  
  \ifx#1\displaystyle\textfont3\else
    \ifx#1\textstyle\textfont3\else
      \ifx#1\scriptstyle\scriptfont3\else
        \scriptscriptfont3\relax
  \fi\fi\fi
}
\makeatother

\newcommand{\surj}{\to\kern-1.8ex\to}

\newcommand{\cP}{\mathcal{P}}

\newcommand{\cA}{\mathcal{A}}
\newcommand{\cK}{\mathcal{K}}
\newcommand{\cB}{\mathcal{B}}

\newcommand{\Conf}{\mathrm{Conf}}
\newcommand{\Sol}{\mathrm{Sol}}

\newcommand{\fre}{\mathfrak{e}}
\newcommand{\frf}{\mathfrak{f}}
\newcommand{\G}{\mathrm{G}}
\newcommand{\cH}{\mathcal{H}}

\newcommand{\mS}{\mathrm{S}}
\newcommand{\cL}{\mathcal{L}}

\newcommand{\ric}{\mathrm{Ric}}

\newcommand{\cI}{\mathcal{I}}
\newcommand{\dd}{\mathrm{d}}
\newcommand{\Cl}{\mathrm{Cl}}

\newcommand{\Gl}{\mathrm{Gl}}

\newcommand{\U}{\mathcal{U}}

\begin{document}

\title[Parallel spinor flows]{Parallel spinor flows on three-dimensional Cauchy hypersurfaces}
 
\author[\'A. Murcia]{\'Angel Murcia$^*$}

\address{Istituto Nazionale di Fisica Nucleare, Sezione di Padova, Repubblica Italiana}
\email{angel.murcia@pd.infn.it}

\author[C. S. Shahbazi]{C. S. Shahbazi} \address{Departamento de Matem\'aticas, Universidad UNED - Madrid, Reino de Espa\~na}
\email{cshahbazi@mat.uned.es} 
\address{Fachbereich Mathematik, Universit\"at Hamburg, Bundesrepublik Deutschland.}
\email{carlos.shahbazi@uni-hamburg.de}

\thanks{$^*$Corresponding author: \texttt{angel.murcia@pd.infn.it} .}

\maketitle

\begin{abstract} 
The three-dimensional parallel spinor flow is the evolution flow defined by a parallel spinor on a globally hyperbolic Lorentzian four-manifold. We prove that, despite the fact that Lorentzian metrics admitting parallel spinors are not necessarily Ricci flat, the parallel spinor flow preserves the vacuum momentum and Hamiltonian constraints and therefore the Einstein and parallel spinor flows coincide on common initial data. Using this result, we provide an initial data characterization of parallel spinors on Ricci flat Lorentzian four-manifolds, which in turn yields the first initial data characterization of Ricci-flat pp-waves. Furthermore, we explicitly solve the left-invariant parallel spinor flow on simply connected Lie groups, obtaining along the way necessary and sufficient conditions for the flow to be immortal. These are, to the best of our knowledge, the first non-trivial examples of evolution flows of parallel spinors. Finally, we use some of these examples to construct families of $\eta\,$-Einstein cosymplectic structures and to produce solutions to the left-invariant Ricci flow in three dimensions. This suggests the intriguing possibility of using first-order hyperbolic spinorial flows to construct special solutions of curvature flows in Riemannian signature.
\end{abstract}



\section{Introduction}
\label{sec:intro}

 
This paper is devoted to the study of the evolution problem posed by a parallel real and irreducible spinor defined on a globally hyperbolic Lorentzian four-manifold $(M,g)$. This problem is well-posed by the results of Leistner and Lichewski, who proved the statement in arbitrary dimension \cite{Lischewski:2015cya,LeistnerLischewski}. Existence of such a parallel spinor field is obstructed since it implies $(M,g)$ to be a solution of Einstein equations with pure radiation type of energy momentum tensor \cite{Mars}. This fact, which translates into a curvature condition on the underlying globally hyperbolic Lorentzian four-manifold, relates the study of parallel real spinors to the study of globally hyperbolic Lorentzian manifolds satisfying a given curvature condition, which has been a fundamental problem in global Lorentzian geometry since the seminal work of G. Mess \cite{Mess,BonsanteSeppi}. 

The starting point of our work is the theory of spinorial polyforms and parabolic pairs, recently proposed in \cite{Cortes:2019xmk}, which allows to study first-order spinorial equations on pseudo-Riemannian manifolds in terms of differential systems for algebraically constrained polyforms, a reformulation which is especially convenient to study global geometric and topological aspects of such equations. The study of spinorial differential systems through differential forms is by now classical in the mathematics and physics literature, see for instance \cite{Graf,BennTucker,Lawson,Harvey,Tod1,Tod2} and their refences and citations for more details. Reference \cite{Cortes:2019xmk} elaborates on these earlier works to provide a bijection between certain irreducible spinors and polyforms based on the precise description of the square of a spinor in terms of a semi-algebraic real body in the K\"ahler-Atiyah bundle of the underlying pseudo-Riemannian manifold. Using the formalism of parabolic pairs \cite{Cortes:2019xmk}, our first result (see Theorem \ref{thm:parallelspinorflow}) reformulates the evolution problem of a parallel spinor as a system of flow equations for a family of functions  $\left\{ \beta_t\right\}_{t\in \cI}$ and a family of coframes $\left\{ \fre^t\right\}_{t\in \cI}$ on an appropriately chosen Cauchy hypersurface $\Sigma \subset M$. This system of equations yields a generalization of \cite[Theorem 5.4]{Murcia:2020zig} to the case in which the family $\left\{\beta_t\right\}_{t\in\cI}$ is non-trivial and defines the notion of \emph{parallel spinor flow} as a family $\left\{\beta_t , \fre^t\right\}_{t\in\cI}$ satisfying equations \eqref{eq:globhyperspinorflowI} and \eqref{eq:globhyperspinorflowII}. Using the notion of parallel spinor flow, the constraint equations of the initial value problem of a parallel spinor can be shown to be equivalent to a differential system, the \emph{parallel Cauchy differential system} \cite{Murcia:2020zig}, for a pair $(\fre,\Theta)$, where $\fre$ is a coframe and $\Theta$ is a symmetric two-tensor on $\Sigma$. The Riemannian metric $h$ induced by $(M,g)$ on $\Sigma$ is given by the canonical metric for which $\fre = (e_u,e_l,e_n)$ becomes an orthonormal coframe, that is, $h_{\fre} = e_u\otimes e_u +e_l\otimes e_l +e_n\otimes e_n$. This allows us to define a natural map:
\begin{equation*}
\Psi\colon \Conf(\Sigma) \to \mathrm{Met}(\Sigma) \times \Gamma(T^{\ast}M\odot T^{\ast}M)\, , \qquad (\fre,\Theta) \mapsto (h_{\fre},\Theta)\, ,
\end{equation*}
 
\noindent
where $\Conf(\Sigma)$ denotes the space of variables $(\fre, \Theta)$. The relevance of this map becomes apparent after observing that $\mathrm{Met}(\Sigma) \times \Gamma(T^{\ast}M\odot T^{\ast}M)$ is precisely the configuration space of the vacuum \emph{Hamiltonian} and \emph{momentum} constraint equations \cite{ChoquetBruhatBook}. This allows to define the notion of initial data \emph{admissible} to both the parallel spinor flow and the vacuum Einstein flow, which immediately leads to the natural question of the compatibility of both flows when starting on common admissible data. We solve this question in the affirmative in Theorem \ref{thm:preservingconstraints}, where we prove that a parallel spinor flow whose initial data are admissible to both problems preserves the Hamiltonian and momentum constraints and produces a Ricci flat Lorentzian four-manifold $(M,g)$. We find this result interesting because it allows us to obtain an \emph{initial data} characterization of parallel spinors on a Ricci flat Lorentzian four-manifold which, to the best of our knowledge, is the first of its type in the literature. More precisely, we prove the following result as a consequence of Theorem \ref{thm:preservingconstraints}.

\begin{corollary}
An initial vacuum data $(\Sigma,h,\Theta)$ admits a Ricci flat Lorentzian development carrying a parallel spinor if and only if there exists a global orthonormal coframe $\fre$ on $\Sigma$ such that $(\fre,\Theta)$ is a parallel Cauchy pair, that is, if and only if $(\fre,\Theta)$ satisfies the following differential system:
\begin{eqnarray*}
\dd e_u=\Theta(e_u)\wedge e_u \, , \quad \dd e_l=\Theta(e_l) \wedge e_u  \, , \quad \dd e_n   = \Theta(e_n) \wedge e_u \, , \quad  [\Theta(e_u) ] = 0 \in H^1(\Sigma,\mathbb{R})\, .\label{eq:exderKIIintro}
\end{eqnarray*}

\noindent
where $\fre = (e_u,e_l,e_n)$.
\end{corollary}  

\noindent
By the well-known local equivalence between \emph{pp-waves} \cite{Mars} and Lorentzian four-manifolds admitting a parallel spinor \cite{EhlersKundt}, the previous corollary provides the first initial data characterization of Ricci-flat pp-waves. More precisely:

\begin{corollary}
An initial vacuum data $(\Sigma,h,\Theta)$ admits a pp-wave Ricci flat Lorentzian development if and only if  $(\Sigma,h)$ admits a parallel Cauchy pair.
\end{corollary} 

\noindent
We also obtain the following corollary.

\begin{corollary}
A globally hyperbolic Lorentzian four-manifold $(M,g)$ admitting a parallel spinor is Ricci flat if and only if there exists a Cauchy hypersurface $\Sigma\subset M$ whose Hamiltonian constraint vanishes.
\end{corollary}

\noindent 
Another important feature of parallel spinorial flows is that they admit a canonical notion of \emph{left-invariance} in terms of which we can define the notion of left-invariance of parallel spinors. The classification of left-invariant admissible initial data on simply connected three-dimensional Lie groups was completed in \cite{Murcia:2020zig}. We elaborate on this result to obtain the classification of all associated left-invariant parallel spinor flows. As expected, the result strongly depends on the initial data $(\fre,\Theta)$ chosen. Given a parallel Cauchy pair $(\fre,\Theta)$ write:
\begin{equation*}
\Theta = \Theta_{ab}\, e_a\otimes e_b\, ,\qquad a, b = u, l, n\, , \qquad \fre = (e_u , e_l , e_n)\, ,
\end{equation*}

\noindent
and define: 
\begin{equation*}
\lambda := \sqrt{\Theta_{ul}^2 + \Theta_{un}^2}\, , \quad \theta := \begin{pmatrix}
	\tll & \tln \\
	\tln & \tnn
\end{pmatrix} \, , \quad T := \mathrm{Tr}(\theta)\, ,  \quad \Delta := \mathrm{Det}(\theta)\, .
\end{equation*}

\noindent
In Section \ref{sec:leftinvariantspinorflow} we prove the following classification result.

\begin{theorem}
\label{thm:hflows}
Let $\{\beta_t, \fre^t\}_{t \in \cI}$ be a left-invariant parallel spinor on a simply-connected Lie group $\G$. Denote by $\left\{ h_{\fre^t} \right\}_{t\in\cI}$ the family of Riemannian metrics associated to $\{\beta_t, \fre^t\}_{t \in \cI}$ and by $(\fre,\Theta)$ its initial parallel Cauchy pair.  
\begin{itemize}[leftmargin=*]
\item If $\tul^2 + \tun^2 = 0$ and $\tuu \neq 0$ then:
\begin{equation*}
h_{\fre^t} =\left (1-\tuu \mathcal{B}_t \right )^2 e_u \otimes e_u 
  + (e_l,e_n)\, Q
	\begin{pmatrix}
		\left[1- \Theta_{uu}\cB_t \right]^{2\rho_{+}} & 0 \\
		0 &  \left[1- \Theta_{uu}\cB_t \right]^{2\rho_{-}}
	\end{pmatrix}  Q^{\ast} \begin{pmatrix}
		e_l \\
		e_n  
	\end{pmatrix}
 \, .
\end{equation*}

\noindent
where $\rho_{\pm}=\frac{T \pm \sqrt{T^2- 4\Delta}}{2\Theta_{uu}}$ are the eigenvalues of $\theta/\Theta_{uu}$ and $Q$ is its orthogonal diagonalization matrix. In particular, $\G= \mathbb{R}^3$ if $\theta = 0$, $\G =\mathrm{E}(1,1)$ if $T=0$ and $\theta\neq 0$, $\G = \tau_2 \oplus \mathbb{R}$ if $T\neq 0$ and $\Delta = 0$ and $\G = \tau_{3,\mu}$ if $T\neq 0$ and $\Delta \neq 0$. The case $\tuu=0$ is obtained by taking the formal limit $\tuu \rightarrow 0$ in the previous expressions.
	
\item If $\tul=0$ but $\tun \neq 0$, we have:
\begin{eqnarray*}
& h_{\fre^t} =\left ( \left (1-\tuu \mathcal{B}_t \right )^2 \sec^2 [y_t]+\frac{\tuu^2}{\lambda^2}- \frac{2\tuu}{\lambda} \left (1-\tuu \mathcal{B}_t \right ) \mathrm{Tan}[y_t] \right  ) e_u \otimes e_u\\
& \left ( \frac{\tuu}{\tun} - \tun \mathcal{B}_t \sec^2[y_t] \left (1-\tuu \mathcal{B}_t 
 \right )- \frac{\lambda}{\tun}\mathrm{Tan}[y_t] \left (1- 2 \tuu\mathcal{B}_t\right  )  \right ) e_u \odot e_n\\
&+e_l \otimes e_l+ \left ( 1 + \lambda^2 \mathcal{B}_t^2 \sec^2[y_t] + 2\,\mathcal{B}_t \lambda \mathrm{Tan}[y_t]\right ) e_n \otimes e_n\, ,
\end{eqnarray*}

\noindent
where $y_t = \lambda \, \cB_{t} + \mathrm{Arctan} \left[ \frac{\tuu}{\lambda}\right]$. In particular $\G = \tau_{2} \oplus \mathbb{R}$. The case $\tul \neq 0$ but $\tun=0$ is obtained by just exchanging the subindices $l$ and $n$ in the previous expression.

\item If $\tul \tun \neq 0$ then:
\begin{eqnarray*}
& h_{\fre^t} = \left ( (1+ T\, \mathcal{B}_t)^2+ \left (\mathrm{Tan}[y_t] (1+ T \mathcal{B}_t) + \frac{T}{\lambda}\right )^2 \right ) e_u \otimes e_u\\
& -\left ( \frac{T}{\lambda^2}+\frac{\mathrm{Tan}[y_t]}{\lambda}(1+2T \cB_t)+ \mathcal{B}_t (1+ T \mathcal{B}_t))  \sec^2 [y_t] \right) (\tul e_u \odot e_l +\tun e_u \odot e_n) \\
& + \left ( 1+ \tul^2 \mathcal{B}_t  \left (\mathcal{B}_t \sec^2 [y_t] +\frac{2 \mathrm{Tan} [y_t]}{\lambda} \right ) \right )e_l \otimes e_l  +\tul \tun \mathcal{B}_t \sec^2[y_t]\left ( \mathcal{B}_t +\frac{\sin[2y_t]}{\lambda} \right ) e_l \odot e_n \\ 
& + \left ( 1+ \tun^2 \mathcal{B}_t  \left (\mathcal{B}_t \sec^2 [y_t] + \frac{2 \mathrm{Tan}[y_t]}{\lambda} \right ) \right )e_n \otimes e_n\, ,
\end{eqnarray*}
		
\noindent
In particular, $\G = \tau_{2} \oplus \mathbb{R}$.
\end{itemize}

\noindent
Furthermore, if $\lambda = 0$ the flow is globally defined (namely $\cI= \mathbb{R}$) if and only if $\int_0^{\infty} \beta_{\tau}\dd\tau < \vert \Theta_{uu}^{-1}\vert$, whereas if $\lambda \neq 0$ the flow is globally defined if and only if $\vert y_t \vert < \frac{\pi}{2}$ $\forall t \in \mathbb{R}$. 
\end{theorem}

\noindent
The notation used above to label simply-connected three-dimensional Lie groups is standard and can be found in \cite[Appendix A]{Freibert}. The previous theorem yields the families of three-dimensional Riemannian metrics that are canonically associated to left-invariant parallel spinor flows. For each such family $\left\{ h_{\fre^t}\right\}_{t\in \cI}$ we obtain a globally hyperbolic Lorentzian four-manifold carrying parallel spinors as follows:
\begin{equation*}
(M,g) = (\cI\times \Sigma, -\beta_t^2 \dd t^2 + h_{\fre^t})\, .
\end{equation*}

\noindent
These globally hyperbolic four-manifolds may or may not be Ricci flat depending on the case considered as explained in Section \ref{sec:leftinvariantspinorflow}. 

The curvature of the families of metrics $\left\{ h_{\fre^t}\right\}_{t\in \cI}$ can be explicitly computed on a case by case basis and a direct inspection of the result reveals that appropriately choosing $\Theta$ we can recover a particular example of a left-invariant Ricci flow as well as a family of $\eta\,$-Einstein cosymplectic structures. It would be interesting to explore the relation between spinorial flows in Lorentzian signature and weakly parabolic flows in Riemannian geometry in more generality and for more complicated types of spinorial equations, especially for those appearing as Killing spinor equations in four-dimensional supergravity theories. Work in this direction is in progress. On the other hand, it is well-known \cite{Mars} that existence of a parallel and irreducible real spinor on $(M,g)$ is locally equivalent to the following curvature condition:
\begin{equation}
\label{eq:introcurvature}
\mathrm{Ric}^g = f\, u\otimes u\, ,
\end{equation}

\noindent
for a local function $f$ and light-like parallel one-form $u$. Therefore, the parallel spinor flow may potentially be locally equivalent to the evolution flow prescribed by equations \eqref{eq:introcurvature}. Such relation does not seem however to be straightforward, since the parallel spinor flow is of \emph{first order} whereas the evolution flow prescribed by \eqref{eq:introcurvature} is of \emph{second order}. This offers an alternative point of view on the parallel spinor flow as a convenient tool to access the evolution flow prescribed by \eqref{eq:introcurvature} in terms of a more transparent and easier to handle first order evolution flow which, intuitively speaking, provides a \emph{first integral} of the evolution flow defined by \eqref{eq:introcurvature}. 

The outline of this paper is as follows. In Section \ref{sec:Spinors4d} we review the theory of parallel spinors in terms of parabolic pairs and we use it to characterize all standard Brinkmann space-times admitting parallel spinors, obtaining a global result in Proposition \ref{prop:standardBrinkmann} that seems to be new in the literature. Section \ref{sec:GlobhyperSpinors4d} is devoted to the description of parallel spinors on globally hyperbolic Lorentzian four-manifolds as parallel spinor flows on an appropriately chosen Cauchy surface. We prove the compatibility of the parallel spinor flow with the vacuum Einstein flow and we obtain an \emph{initial data} characterization of parallel spinors on Ricci flat Lorentzian four-manifolds. Finally, in Section \ref{sec:leftinvariantspinorflow} we classify all left-invariant parallel spinor flows on simply connected three-dimensional Lie groups and we elaborate on some of their properties. 


\subsection*{Acknowledgements/Funding}


We would like to thank Vicente Cort\'es and Miguel S\'anchez for very interesting discussions and comments. The work of \'A.M. was  funded by the Spanish FPU Grant No. FPU17/04964, by the Deutscher Akademischer Austauschdienst (DAAD), through the Short-Term Research Grant No. 91791300, and by the Istituto Nazionale di Fisica Nucleare (INFN) through the INFN Call No. 23590. \'A.M. received additional support from the MCIU/AEI/FEDER UE grant PGC2018-095205-B-I00 and the Centro de Excelencia Severo Ochoa Program grant SEV-2016-0597. The work of C.S.S. was supported by the Germany Excellence Strategy \emph{Quantum Universe} - 390833306.
 
\subsection*{Author contribution}

All authors contributed equally.

\subsection*{Conflicts of interest/Competing interests statement}

The authors have no conflicts of interest to declare that are relevant to the content of this article.

\subsection*{Data Availability Statement}

Data sharing not applicable to this article as no datasets were generated or analysed during the current study.


\section{Parallel spinors on Lorentzian four-manifolds}
\label{sec:Spinors4d}


In this section we first review the theory of parallel spinors on globally hyperbolic Lorentzian four-dimensional manifolds and, afterwards, we introduce the notion of standard Brinkmann spacetimes and study when they admit a parallel spinor.


\subsection{General theory}
\label{subsec:GeneralSpinors4d}

We begin by briefly presenting the theory of parallel spinors on globally hyperbolic Lorentzian four-dimensional manifolds as developed in \cite{Cortes:2019xmk,Murcia:2020zig}, where parallel spinors were described in terms of a specific type of distribution satisfying a certain system of partial differential equations. Let $(M,g)$ be a four-dimensional \emph{space-time}, i.e. a connected, oriented and time oriented four-manifold endowed with a Lorentzian metric $g$. We assume that $(M,g)$ admits a bundle of irreducible real spinors $\mathrm{S}_g$. The existence of such $\mathrm{S}_g$ is equivalent \cite{LazaroiuShahbazi,Lazaroiu:2017zpl,Lazaroiu:2018igs} to the existence of a spin structure $Q_g$, in which case $\mS_g$ can be conceived as the vector bundle associated to $Q_g$ via the tautological representation given by the canonical embedding $\mathrm{Spin}_{+}(3,1)\subset \Cl(3,1)$, where $\mathrm{Spin}_{+}(3,1)$ stands for the connected component of the identity of the spin group in signature $(3,1) = -+++$ and $\Cl(3,1)$ denotes the real Clifford algebra in the aforementioned signature. Therefore, we will assume that $(M,g)$ is spin and endowed with a fixed spin structure $Q_g$. Under these conditions, the Levi-Civita connection $\nabla^g$ on $(M,g)$ induces naturally a connection on $\mS_g$, the \emph{spinorial} Levi-Civita connection, which for the sake of simplicity is denoted by the same symbol. 

\begin{definition}
A \emph{spinor} $\varepsilon$ on $(M,g,\mS_g)$ is a smooth section $\varepsilon\in \Gamma(\mS_g)$ of $\mS_g$. It is \emph{parallel} if $\nabla^g\varepsilon = 0$.
\end{definition}

\noindent
Let $u\in \Omega^1(M)$ be a luminous one-form. We declare two one-forms $l_1 , l_2 \in \Omega^1(M)$ to be equivalent through the equivalence relation $\sim_u$, $l_1 \sim_u l_2$, if and only if $l_1 = l_2 + f u$ with $f\in C^{\infty}(M)$. We denote by:
\begin{equation*}
\Omega^1_{u}(M) := \frac{\Omega^1(M)}{\sim_u}\, ,
\end{equation*}

\noindent
the $C^\infty(M)$-module of equivalence classes defined by $\sim_u$.

\begin{definition}
A parabolic pair $(u,[l])$ on $(M,g)$ is conformed by a nowhere vanishing null one-form $u\in \Omega^1(M)$ and an equivalence class of one-forms $[l]\in \Omega^1_u(M)$ such that:
\begin{equation*}
g(l,u) =  0\, , \qquad g(l,l) = 1\, ,
\end{equation*}
	
\noindent
for any, and hence, for all, representatives $l\in [l]$.
\end{definition}

\noindent
The one-form $u$ is usually called the \emph{Dirac current} of the spinor $\varepsilon$. From \cite[Theorems 4.26 and 4.32]{Cortes:2019xmk}, it is possible to characterize parallel spinors on $(M,g)$ in terms of differential equations for parabolic pairs. 
\begin{proposition}
\label{prop:parallelspinorgeneral}
A space-time $(M,g)$ admits a parallel spinor $\varepsilon\in \Gamma(\mS_g)$ if and only if there exists a parabolic pair $(u,[l])$ on $(M,g)$ satisfying:
\begin{equation}
\label{eq:parallelspinorgeneral}
\nabla^g u = 0\, , \qquad \nabla^g l = \kappa\otimes u\, ,
\end{equation}
	
\noindent
where $\kappa \in \Omega^1(M)$.  
\end{proposition}

\begin{remark}
Specifically, \cite[Theorem 4.26]{Cortes:2019xmk} states the equivalence between particular classes of first-order partial differential equations for $\varepsilon$ and certain systems of partial differential equations for $(u,[l])$, of which Equations \eqref{eq:parallelspinorgeneral} represent their simplest case. 
\end{remark}

\noindent
A parabolic pair $(u,[l])$ is \emph{parallel} if Equations \eqref{eq:parallelspinorgeneral} are satisfied for a representative $l\in [l]$.


\subsection{Standard Brinkmann space-times}


In order to illustrate the various uses of Proposition \ref{prop:parallelspinorgeneral} and make contact with the existing literature, in this subsection we recover the well-known local characterization of a Lorentzian four-manifold $(M,g)$ admitting a parallel spinor, obtaining along the way the global characterization of \emph{standard} Brinkmann space-times that admit a parallel spinor, which seems to be new in the literature. Recall that by definition a \emph{Brinkmann} space-time \cite{Brinkmann,Mars} is a Lorentzian four manifold equipped with a parallel null vector. Let $(u,[l])$ be a parallel parabolic pair on $(M,g)$, which by Proposition \ref{prop:parallelspinorgeneral} is equivalent to the existence of a parallel spinor. Since $u$ is parallel, $(M,g)$ is locally isometric to a Brinkmann space-time, whence it suffices to consider $(M,g)$ to be \emph{standard}, namely $M = \mathbb{R}^2\times X$ in terms of an oriented two-dimensional manifold $X$, equipped with the metric:
\begin{equation*}
g = H_{x_{u}} \dd x_u \otimes \dd x_u + \dd x_{u}\odot \alpha_{x_u}  + \dd x_u\odot \dd x_v + q_{x_u}\, .
\end{equation*} 

\noindent
where $(x_u,x_v)$ denotes the Cartesian coordinates of $\mathbb{R}^2$, and:
\begin{equation*}
\left\{ H_{x_{u}} \right\}_{x_{u}\in \mathbb{R}}\, , \quad \left\{\alpha_{x_u} \right\}_{x_{u}\in \mathbb{R}}\, , \quad \left\{ q_{x_u}\right\}_{x_{u}\in \mathbb{R}}\, ,
\end{equation*}

\noindent
respectively denote a family of functions, a family of one-forms and a family of complete Riemannian metrics on $X$ parametrized by $x_{u}\in \mathbb{R}$. The vector field $\partial_{x_v}$ is null and parallel, so $g(\partial_{x_v}) = \dd x_u$ is a null parallel one-form which we identify with $u$. We will refer to a parallel spinor $\varepsilon$ on a standard Brinkmann space-time as adapted if its Dirac current is $u^{\sharp_g} = \partial_{x_v}$. In such case, the first equation in \eqref{eq:parallelspinorgeneral} is automatically satisfied and we only need to be concerned with the second equation in \eqref{eq:parallelspinorgeneral}, namely:
\begin{equation*}
\nabla^g l = \kappa \otimes \dd x_u\, , \qquad l\in [l]\, , \qquad \kappa \in \Omega^1(M)\, ,
\end{equation*}

\noindent
which needs to be satisfied for a representative in $[l]$. This equation is equivalent to:
\begin{equation}
\label{eq:lsplit}
\dd l = \kappa\wedge \dd x_u\, , \qquad   \cL_{l^\sharp} g = \kappa \odot \dd x_u\, , 
\end{equation}

\noindent
where $l^{\sharp_g}$ denotes the metric dual of $l$ with respect to $g$. Using that $u = \dd x_u$ and $g^{-1}(l,u) =l(\partial_{x_v}) = 0$, it follows that there exists a representative $l\in [l]$ of the form:
\begin{equation*}
l =  l^{\perp}\, ,
\end{equation*}

\noindent
where $l^{\perp} $ denotes a bi-parametric family of unit-norm one-forms on $X$ parametrized by $(x_u,x_v) \in \mathbb{R}^2$. The first equation in \eqref{eq:lsplit} is equivalent to:
\begin{equation*}
\kappa_v \,\dd x_u \wedge \dd x_v - \partial_{x_v} l^{\perp} \wedge \dd x_v - (\kappa^{\perp} + \partial_{x_u} l^{\perp} ) \wedge \dd x_u +  \dd_X l^{\perp}  = 0\, ,
\end{equation*}	
 
\noindent
where we have written $\kappa = \kappa_u \dd x_u + \kappa_v \dd x_v + \kappa^{\perp}$ and $\dd_X$ denotes the exterior derivative operator on $X$. The general solution to the previous equation reads:
\begin{equation*}
\kappa_v = 0\, , \qquad  \partial_{x_v} l^{\perp}= 0\, , \qquad \kappa^{\perp} = - \partial_{x_u} l^{\perp}\, , \qquad \dd_X l^{\perp} = 0\, .
\end{equation*}

\noindent
Therefore, since $\partial_{x_v} l^{\perp}= 0$, the one-form $l^{\perp}$ is equivalent to a family $\left\{\ell_{x_u}\right\}_{x_u\in\mathbb{R}}$ of unit-norm one-forms on $X$. Recall now that the dual $l^{\sharp}$ of $l$ with respect to $g$ is given by the following expression:
\begin{equation*}
l^{\sharp} =  - \alpha_{x_u}(\ell^{\sharp_{q}}_{x_u}) \, \partial_{x_v} + \ell^{\sharp_{q}}_{x_u}  
\end{equation*}

\noindent
where the symbol $\sharp_{q}$ denotes musical isomorphism with respect to the metric $q_{x_u}$. We use this equation to expand the Lie derivative of $g$ along $l^{\sharp}$ as follows:
\begin{eqnarray*}
& \cL_{l^\sharp}g = \dd H_{x_u}(\ell^{\sharp_{q}}_{x_u}) \,\dd x_u\otimes \dd x_u  + \dd x_u \odot ( \alpha_{x_u}(\partial_{x_u}\ell^{\sharp_{q}}_{x_u}) \dd x_u + \cL_{\ell^{\sharp_{q}}_{x_u}}^X \alpha_{x_u})  - \dd x_u\odot \dd(\alpha_{x_u}(\ell^{\sharp_{q}}_{x_u}))  + \cL_{\ell^{\sharp_{q}}_{x_u}}^X q_{x_u}
\end{eqnarray*}

\noindent
where $\cL^X$ denotes the Lie derivative on the surface $X$ and where we have used:
\begin{eqnarray*}
& \cL_{l^\sharp} \dd x_u = 0\, , \quad \cL_{l^\sharp} \dd x_v = - \dd (\alpha_{x_u}(\ell^{\sharp_{q}}_{x_u})) = - \partial_{x_u} (\alpha_{x_u}(\ell^{\sharp_{q}}_{x_u})) \dd x_u - \dd_X (\alpha_{x_u}(\ell^{\sharp_{q}}_{x_u}))\, , \\
& \cL_{l^\sharp} \alpha_{x_u} = \alpha_{x_u}(\partial_{x_u}\ell^{\sharp_{q}}_{x_u}) \dd x_u + \cL_{\ell^{\sharp_{q}}_{x_u}}^X \alpha_{x_u}\, , \quad  \cL_{l^\sharp} q_{x_u} =  \cL_{\ell^{\sharp_{q}}_{x_u}}^X q_{x_u}+\dd x_u \odot (\partial_{x_u} \ell_{x_u}- (\partial_{x_u} q_{x_u})(\ell_{x_u}^{\sharp_{q}})) \, .
\end{eqnarray*}

\noindent
Hence, the second equation in \eqref{eq:lsplit} is equivalent to:
\begin{eqnarray}
\label{eq:standardax1}
& \kappa_u = \frac{1}{2} (\dd H_{x_u}(\ell_{x_u}) - (\partial_{x_u} \alpha_{x_u})(\ell_{x_u}))\, , \qquad  \nabla^{q_{x_u}} \ell_{x_u} = 0\, ,\\
& 2\partial_{x_u} \ell_{x_u}-(\partial_{x_u} q_{x_u})((\ell_{x_u})^{\sharp_{q}}) + \ell_{x_u} \lrcorner \,\dd_X \alpha_{x_u} = 0\, ,
\label{eq:standardax2}
\end{eqnarray}

\noindent
where we have used that $\dd_X \ell_{x_u} =0$ and where $\nabla^{q_{x_u}}$ denotes the Levi-Civita connection of the Riemannian metric $q_{x_u}$ on $X$. The first equation in \eqref{eq:standardax1} simply determines $\kappa_u$ in terms of the underlying metric and the given adapted parabolic pair, whereas the second one establishes that $\ell_{x_u}$ is a nowhere vanishing parallel one-form with respect to $q_{x_u}$ whence we obtain a parallel global orthonormal frame $\{ \ell_{x_u}, n_{x_u}\}_{x_u\in \mathbb{R}}$ on $X$ by setting:
\begin{equation*}
n_{x_u} = \ast_{q_{x_u}} \ell_{x_u}\, .
\end{equation*}

\noindent
In particular, $(X,q_{x_u})$ is a flat Riemann surface for every $x_u\in \mathbb{R}$ and therefore isometric to either the euclidean plane, a flat cylinder or a flat torus. Projecting Equation \eqref{eq:standardax2} along $\ell_{x_u}$ we obtain an identity, whereas projecting along $n_{x_u}$ we obtain that Equation \eqref{eq:standardax2} is equivalent to:
\begin{equation*}
\dd_X \alpha_{x_u}  =  \partial_{x_u}n_{x_u} \wedge n_{x_u}  + \partial_{x_u} \ell_{x_u} \wedge \ell_{x_u} \, .
\end{equation*}

\noindent
All together, the previous discussion proves the following result.

\begin{proposition}
\label{prop:standardBrinkmann}
A standard Brinkmann space-time admits an adapted parallel spinor if and only if it is isometric to the following model:
\begin{equation*}
(M,g) = (\mathbb{R}^2\times X, H_{x_{u}} \dd x_u \otimes \dd x_u + \dd x_{u}\odot \alpha_{x_u}  + \dd x_u\odot \dd x_v + q_{x_u})\, ,
\end{equation*}

\noindent
and there exists a family of parallel orthonormal coframes $\left\{ \ell_{x_u} , n_{x_u} \right\}_{x_u\in \mathbb{R}}$ on $(X,q_{x_u})$ such that:
\begin{equation*}
\dd_X \alpha_{x_u}  =  \partial_{x_u}n_{x_u} \wedge n_{x_u}  + \partial_{x_u} \ell_{x_u} \wedge \ell_{x_u}\, .
\end{equation*}

\noindent
In particular, $\left\{ q_{x_u}\right\}_{x_u\in \mathbb{R}}$ is a family of complete flat metrics on $X$.  
\end{proposition}

\begin{remark}
Note that the family of functions $\left\{H_{x_u}\right\}_{x_u\in \mathbb{R}}$ does not play any role regarding the existence of an adapted parallel spinor on a standard Brinkmann space-time.
\end{remark}

\noindent
By uniformization, we conclude that $X$ is diffeomorphic to either $\mathbb{R}^2$, $\mathbb{R}^2\backslash\left\{ 0\right\}$ or $T^2$. Appropriately choosing local coordinates the previous result immediately implies that a four-dimensional space-time admitting parallel spinors is locally isometric to a \emph{pp-wave} \cite{EhlersKundt,Mars}, defined as a Brinkmann space for which the Ricci tensor takes the form \eqref{eq:introcurvature}. Note however that  Proposition \eqref{prop:standardBrinkmann} is more general than this local equivalence since Proposition \eqref{prop:standardBrinkmann} is a global result that provides a classification of the topology type of manifolds of the form $M=\mathbb{R}^2\times X$ that admit adapted parallel spinors and gives a global characterization of the corresponding Lorentzian metrics in terms of flow equations on $X$.

\begin{example}
\label{ep:Brinkmann}
Let $(M,g)$ be a simply-connected standard Brinkmann space-time admitting an adapted parallel spinor. Using the notation of Proposition \eqref{prop:standardBrinkmann}, choose families of constant functions $\left\{ c^1_{x_u} , c^2_{x_u}, f_{x_u} \right\}_{x_u\in\mathbb{R}}$ on $X = \mathbb{R}^2$ in terms of which we define the following families of one-forms on $X$:
\begin{equation*}
\ell_{x_u} = e^{c^1_{x_u}} \dd y_1 + f_{x_u} \dd y_2 \, , \qquad n_{x_u} = e^{c^2_{x_u}} \dd y_2\, , \qquad x_u\in \mathbb{R}\, ,
\end{equation*}

\noindent
where $(y_1,y_2)$ are the global Cartesian coordinates of $\mathbb{R}^2$. In particular:
\begin{equation*}
q_{x_u} = \ell_{x_u} \otimes \ell_{x_u} + n_{x_u}\otimes n_{x_u} = e^{2 c^1_{x_u}} \dd y_1 \otimes \dd y_1 + f_{x_u} e^{c^1_{x_u}} \dd y_1\odot \dd y_2 + (e^{2 c^2_{x_u}} + f^2_{x_u}) \dd y_2 \otimes \dd y_2\, .
\end{equation*} 

\noindent
A quick computation shows that:
\begin{equation*}
\partial_{x_u}\ell_{x_u} = \partial_{x_u} c^1_{x_u} \ell_{x_u} + (\partial_{x_u} f_{x_u} - \partial_{x_u} c^1_{x_u} f_{x_u}) e^{-c^2_{x_u}} n_{x_u}\, , \quad \partial_{x_u} n_{x_u} = \partial_{x_u}c^2_{x_u} e^{-c^2_{x_u}} n_{x_u}\, , \quad x_u\in \mathbb{R}\, ,
\end{equation*}

\noindent
whence:
\begin{equation*}
\dd_X \alpha_{x_u}  =  (\partial_{x_u} c^1_{x_u} f_{x_u} - \partial_{x_u} f_{x_u}) \ell_{x_u} \wedge \dd y_2\, .
\end{equation*} 

\noindent
Solutions to this equation can be easily found by direct inspection. A particular solution is given by:
\begin{equation*}
\alpha_{x_u} = (\partial_{x_u} f_{x_u} - \partial_{x_u} c^1_{x_u} f_{x_u})y_2\, \ell_{x_u}\, ,
\end{equation*}

\noindent
which gives the following metric $g$ on $\mathbb{R}^4$:
\begin{eqnarray*}
& g =   H_{x_{u}} \dd x_u \otimes \dd x_u + (\partial_{x_u} f_{x_u} - \partial_{x_u} c^1_{x_u} f_{x_u})\, y_2\, \dd x_{u}\odot (e^{c^1_{x_u}} \dd y_1 + f_{x_u} \dd y_2) + \dd x_u\odot \dd x_v\\
& + e^{2 c^1_{x_u}} \dd y_1 \otimes \dd y_1 + f_{x_u} e^{c^1_{x_u}} \dd y_1\odot \dd y_2 + (e^{2 c^2_{x_u}} + f^2_{x_u}) \dd y_2 \otimes \dd y_2\, .
\end{eqnarray*}

\noindent
This produces an explicit example of Lorentzian metric on $\mathbb{R}^4$ that admits parallel spinors and for which the \emph{crossed term} $\alpha_{x_u}$ is non-trivial and, in particular, not closed. 

\noindent
On the other hand, if $f_{x_u} = e^{c^1_{x_u}}$ we have $\dd_X \alpha_{x_u} =0$ and therefore $\alpha_{x_u} = \dd_X \mathfrak{o}_{x_u}$ for a family $\left\{\mathfrak{o}_{x_u}\right\}_{x_u\in \mathbb{R}}$ of functions on $\mathbb{R}^2$. Hence, in this case the metric $g$ can be written as follows:
\begin{eqnarray}
& g =  (H_{x_{u}} - 2 \partial_{x_u}\mathfrak{o}_{x_u}) \dd x_u \otimes \dd x_u + \dd x_{u}\odot \dd (\mathfrak{o}_{x_u}  +  x_v) \nonumber \\ 
& + e^{2 c^1_{x_u}} \dd y_1 \otimes \dd y_1 +  e^{2 c^1_{x_u}} \dd y_1\odot \dd y_2 + (e^{2 c^1_{x_u}} + e^{2 c^2_{x_u}}) \dd y_2 \otimes \dd y_2\, .
\label{eq:metricfinalexample}
\end{eqnarray}

\noindent
and therefore we can write locally:
\begin{equation*}
g =  \hat{H}_{x_{u}}  \dd x_u \otimes \dd x_u + \dd x_{u}\odot \dd \hat{x}_{v} + e^{2 c^1_{x_u}} \dd y_1 \otimes \dd y_1 +  e^{2 c^1_{x_u}} \dd y_1\odot \dd y_2 + (e^{2 c^1_{x_u}} + e^{2 c^2_{x_u}}) \dd y_2 \otimes \dd y_2\, ,
\end{equation*}

\noindent
where we have set $\hat{H}_{x_u} = H_{x_{u}} - 2 \partial_{x_u}\mathfrak{o}_{x_u}$ and $\hat{x}_v = \mathfrak{o}_{x_u}  +  x_v$.
\end{example}

\noindent
\begin{remark}
By applying a $x_u$-dependent family of diffeomorphisms $\left\{ f_{x_u}\right\}_{x_u\in \mathbb{R}}$ of $X$, we can remove the \emph{crossed-term} $\dd x_u\odot \alpha_{x_u}$ in the metric $g$ at least locally, obtaining a metric of the type:
\begin{equation*}
f^{\ast}_{x_u}g =  H_{x_{u}}\circ f_{x_u} \dd x_u \otimes \dd x_u + \dd x_u\odot \dd x_v + f_{x_u}^{\ast} q_{x_u}\, .
\end{equation*}

\noindent
However, in many situations it may not be convenient to implement this change of coordinates, since in general it will spoil any choice of coordinates in which the family of metrics $\left\{q_{x_u}\right\}_{x_u\in\mathbb{R}}$ adopts a simple form.
\end{remark}


\section{Globally hyperbolic case}
\label{sec:GlobhyperSpinors4d}

 
Let $(M,g)$ be a globally hyperbolic four-dimensional space-time. By the results of Bernal and S\'anchez \cite{Bernal:2003jb,Bernal:2004gm}, $(M,g)$ turns out to have the following isometry type:
\begin{equation}
\label{eq:globahyp}
(M,g) = (\cI\times \Sigma, -\beta^2_t \dd t\otimes \dd t + h_t)\, ,
\end{equation}

\noindent
where $t$ is the canonical coordinate on the interval $\cI\subset \mathbb{R}$, $\left\{\beta_t\right\}_{t\in \cI}$ is a smooth family of strictly positive functions on $\Sigma$ and $\left\{ h_t\right\}_{t\in \cI}$ is a family of complete Riemannian metrics on $\Sigma$. From now on we fix the identification \eqref{eq:globahyp} and we set:
\begin{equation*}
\Sigma_t := \left\{ t\right\}\times \Sigma \hookrightarrow M\, , \qquad \Sigma := \left\{ 0\right\}\times \Sigma \hookrightarrow M\, ,
\end{equation*}

\noindent
and define:
\begin{equation*}
	\mathfrak{t}_t = \beta_t\, \dd t\, ,
\end{equation*}

\noindent
to be the (outward-pointing) unit time-like one-form orthogonal to $T^{\ast}\Sigma_t$ for every $t\in \mathbb{R}$. We equip $\Sigma\hookrightarrow M$  with the induced Riemannian metric:
\begin{equation*}
	h := h_0\vert_{T\Sigma\times T\Sigma}\, ,
\end{equation*}

\noindent
and we consider $(\Sigma, h)$ to be the Cauchy hypersurface of $(M,g)$. Associated to each embedded manifold  $\Sigma_t\hookrightarrow M$ we have its \emph{shape operator} or scalar second fundamental form $\{\Theta_t\}_{t \in \mathcal{I}}$ which is defined in the usual way:
\begin{equation*}
	\Theta_t  := \nabla^g \mathfrak{t}_t\vert_{T\Sigma_t\times T\Sigma_t}\, ,
\end{equation*}
 
\noindent
or, equivalently:
\begin{equation*}
	\Theta_t = - \frac{1}{2\beta_t} \partial_t h_t \in \Gamma(T^{\ast}\Sigma_t\odot T^{\ast}\Sigma_t)\, .
\end{equation*}

\noindent
We will denote the shape operator of $\Sigma_t\subset M$ either by $\Theta_t$ or $\Theta^t$, depending on convenience. It can be checked that:
\begin{equation*}
\nabla^g \alpha \vert_{T\Sigma_t\times TM} = \nabla^{h_t} \alpha + \Theta_t(\alpha)\otimes \mathfrak{t}_t\, , \qquad \forall\,\,\alpha\in \Omega^1(\Sigma_t)\, ,
\end{equation*}

\noindent
where $\nabla^{h_t}$ is the Levi-Civita connection on $(\Sigma_t,h_t)$ and $\Theta_t(\alpha) := \Theta_t(\alpha^{\sharp_{h_t}})$ is defined as the evaluation of $\Theta_t$ on the metric dual of $\alpha$. If $(u,[l])$ is a parabolic pair, we set:
\begin{equation*}
	u = u^0_t\, \mathfrak{t}_t + u^{\perp}_t\, , \qquad l = l^0_t\, \mathfrak{t}_t + l^{\perp}_t\in [l]\, ,
\end{equation*}

\noindent
where the superscript $\perp$ indicates orthogonal projection to $T^{\ast}\Sigma_t$ and we have defined:
\begin{equation*}
u^0_t = - g(u,\mathfrak{t}_t)\, , \qquad l^0_t = - g(l,\mathfrak{t}_t)\, .
\end{equation*}

\noindent
Using this orthogonal splitting we describe parallel spinors on a globally hyperbolic space-time in terms of \emph{tensorial} flow equations on $\Sigma$. 

\begin{lemma}\cite[Lemma 2.6]{Murcia:2020zig}
\label{lemma:spinorglobal}
A globally hyperbolic four-manifold $(M,g)  = (\mathbb{R}\times \Sigma, -\beta^2_t \dd t\otimes \dd t + h_t)$ admits a parallel parabolic pair, and hence a parallel spinor field, if and only if there exists a family of orthogonal one-forms $\left\{ u^{\perp}_t, l^{\perp}_t \right\}_{t\in \cI}$ on $\Sigma$ satisfying the following equations:
\begin{eqnarray}
\label{eq:globalesp1}
& \partial_t u^{\perp}_t + \beta_t \Theta_t(u^{\perp}_t) = u^0_t\dd \beta_t\, , \qquad  u^0_t \partial_t l^{\perp}_t +\beta_t  u^0_t \Theta_t(l^{\perp}_t) + \dd \beta_t (l^{\perp}_t) u^{\perp}_t = 0\, ,   \\
& \nabla^{h_t} u^{\perp}_t + u^0_t \Theta_t = 0\, , \qquad u^0_t \nabla^{h_t} l^{\perp}_t = \Theta_t(l^{\perp}_t)\otimes u^{\perp}_t\, , \label{eq:globalesp2}
\end{eqnarray}
	
\noindent
as well as:
\begin{equation}
\label{eq:globalesp3}
(u^0_t)^2 = \vert u^{\perp}_t \vert^2_{h_t}\, , \qquad \vert l^{\perp}_t\vert^2_{h_t} = 1\, ,
\end{equation}
	
\noindent
In particular, $\partial_t u^0_t = \dd \beta_t(u^{\perp}_t)$ and $\dd u^0_t + \Theta_t(u^{\perp}_t) = 0$. If these equations are satisfied, the corresponding parabolic pair $(u,[l])$ is given by:
\begin{equation*}
u = u^0_t \mathfrak{t}_t +  u_t^{\perp}\, ,  \qquad  [l] = [ l_t^{\perp}]\, .
\end{equation*}
	
\noindent
where $\vert u_t^{\perp}\vert^2_{h_t} = h_t(u^{\perp}_t , u^{\perp}_t)$ and $\vert l_t^{\perp}\vert^2_{h_t} = h_t(l^{\perp}_t , l^{\perp}_t)$.
\end{lemma}

\noindent
The previous lemma gives the necessary and sufficient conditions for a globally hyperbolic Lorentzian four manifold $(M,g)$ to admit a real parallel spinor field. We will consider as variables of these equations tuples of the form:
\begin{equation}
\label{eq:originaltuple}
(\left\{ \beta_t \right\}_{t\in \cI} , \left\{ h_t \right\}_{t\in \cI},\left\{ u^0_t\right\}_{t\in \cI},\left\{ u^{\perp}_t\right\}_{t\in \cI} ,\left\{ l^{\perp}_t\right\}_{t\in \cI})\, .
\end{equation}

\noindent
These tuples contain the information about both the real parallel spinor and the underlying globally hyperbolic Lorentzian metric. The following theorem generalizes \cite[Theorem 5.4]{Murcia:2020zig} to the case in which $\left\{ \beta_t\right\}_{t\in\cI}$ is not necessarily constant. 

\begin{theorem} 
\label{thm:parallelspinorflow}
An oriented globally hyperbolic Lorentzian four-manifold $(M,g)$ admits a parallel spinor field if and only if there exists an orientation preserving diffeomorphism identifying $M = \mathcal{I}\times \Sigma$, where $\Sigma$ is an oriented three-manifold equipped with a family of strictly positive functions $\left\{ \beta_t\right\}_{t\in \cI}$ on $\Sigma$ and a family $\left\{ \fre^t \right\}_{t\in \cI}$ of sections of $\mathrm{F}(\Sigma)$ satisfying the following system of differential equations:
\begin{eqnarray}
\label{eq:globhyperspinorflowI}
& \partial_t e^t_a + \dd\beta_t(e^t_a) e^t_u  + \beta_t \Theta_t(e^t_a) =  \delta_{au} \dd\beta_t\, , \qquad \dd \fre^t= \Theta_t(\fre^t) \wedge e^t_u\, ,\\
&   \partial_t(\Theta_t(e^t_u ))  + \dd (\dd \beta_t(e_u^t)) = 0\, ,  \qquad  [\Theta_t(e^t_u )]=0\in H^1(\Sigma,\mathbb{R})\label{eq:globhyperspinorflowII} 
\end{eqnarray}
	
\noindent
where $\fre^t = (e_u^t , e_l^t, e_n^t) \colon \Sigma \to \mathrm{F}(\Sigma)$ and:
\begin{equation*}
h_{\fre^t} = e_u^t\otimes e_u^t + e_l^t\otimes e_l^t + e_n^t\otimes e_n^t\, , \qquad \Theta_t = - \frac{1}{2\beta_t}\partial_t h_{\fre^t}\, .
\end{equation*}
	
\noindent
In this case, the globally hyperbolic metric $g$ is given by:
\begin{equation*}
g = - \beta_t^2 \dd t\otimes \dd t + h_{\fre^t} \, ,
\end{equation*} 
	
\noindent
where $t$ is the Cartesian coordinate in the splitting $M = \mathbb{R}\times \Sigma$.
\end{theorem}

\begin{proof}
By Lemma \ref{lemma:spinorglobal}, a globally hyperbolic Lorentzian four-manifold $(M,g)$ admits a real Killing spinor field if and only if there exists a Cauchy surface $\Sigma \hookrightarrow M$ equipped with a tuple \eqref{eq:originaltuple} satisfying equations \eqref{eq:globalesp1}, \eqref{eq:globalesp2} and \eqref{eq:globalesp3}. Let:
\begin{equation}
(\left\{ \beta_t \right\}_{t\in \cI} , \left\{ h_t \right\}_{t\in \cI},\left\{ u^0_t\right\}_{t\in \cI},\left\{ u^{\perp}_t\right\}_{t\in \cI} ,\left\{ l^{\perp}_t\right\}_{t\in \cI})\, ,
\end{equation}

\noindent
be such a solution and define:
\begin{equation}
\label{eq:redfeuel}
e_u^t = \frac{u^{\perp}_t}{u^0_t}\, , \qquad e^t_l = l^{\perp}_t\, .
\end{equation}

\noindent
Then $(e^t_u , e^t_l)$ is a family of nowhere vanishing and orthonormal one-forms on $\Sigma$, which can be canonically completed to a family of orthonormal coframes $\left\{  \fre^t = (e^t_u , e^t_l , e^t_n) \right\}$ by defining the family of one-forms $\left\{ e_n^t\right\}_{t\in\cI}$ as follows:
\begin{equation*}
e_n^t := \ast_{h_t}(e^t_u \wedge e^t_l)\, ,
\end{equation*} 

\noindent
where $\ast_{h_t}$ denotes the Hodge dual associated to the family of metrics $\left\{ h_t\right\}_{t\in\cI}$. Plugging equations \eqref{eq:redfeuel} into the first and second equations in \eqref{eq:globalesp1} and manipulating the time derivative of $\left\{ u^{\perp}_t\right\}_{t\in \cI}$ we obtain the first equation in \eqref{eq:globhyperspinorflowI} for $a =u$ and $a = l$. For $\left\{ e_n^t\right\}_{t\in\cI}$ we compute as follows:
\begin{eqnarray*}
& 0 = \partial_t (h^{-1}_t (e_n^t, e_u^t )) = (\partial_t  h^{-1}_t) (e_n^t, e_u^t ) + h^{-1}_t (\partial_t e_n^t, e_u^t ) + h^{-1}_t (e_n^t, \partial_t e_u^t ) = 2\beta_t \Theta_t(e_n^t, e_u^t )\\
& + h^{-1}_t (\partial_t e_n^t, e_u^t ) + h^{-1}_t (e_n^t,    \dd\beta_t - \beta_t \Theta_t(e^t_u)) =  \beta_t \Theta_t(e_n^t, e_u^t )+ h^{-1}_t (\partial_t e_n^t, e_u^t ) + \dd\beta_t(e_n^t)\, ,\\
& 0 = \partial_t (h^{-1}_t (e_n^t, e_l^t )) = (\partial_t  h^{-1}_t) (e_n^t, e_l^t ) + h^{-1}_t (\partial_t e_n^t, e_l^t ) + h^{-1}_t (e_n^t, \partial_t e_l^t ) = 2\beta_t \Theta_t(e_n^t, e_u^t )\\
& + h^{-1}_t (\partial_t e_n^t, e_l^t ) + h^{-1}_t (e_n^t,  \beta_t \Theta_t(e^t_l)) =  \beta_t \Theta_t(e_n^t, e_u^t )+ h^{-1}_t (\partial_t e_n^t, e_l^t )\, ,\\
\end{eqnarray*}

\noindent
which immediately implies the first equation in \eqref{eq:globhyperspinorflowI} for the remaining case $a = n$. On the other hand, by Lemma \ref{lemma:spinorglobal}, we have:
\begin{equation*}
\frac{\dd u^0_t}{u^0_t} + \Theta_t(e^{t}_u) = 0\, ,
\end{equation*}

\noindent
whence $[\Theta_t(e^t_u )]=0\in H^1(\Sigma,\mathbb{R})$, which yields the second equation in \eqref{eq:globhyperspinorflowII}. Taking the time derivative of the previous equations we obtain:
\begin{equation*}
\dd \partial_t \log\vert u^0_t\vert + \partial_t(\Theta_t(e^{t}_u)) = 0\, .
\end{equation*}

\noindent
Since by Lemma \ref{lemma:spinorglobal} we have $\partial_t u^0_t = \dd\beta_t(u^{\perp}_t)$, the previous equation implies the first equation in \eqref{eq:globhyperspinorflowII}. We compute:
\begin{equation*}
\nabla^{h_t} e^t_n = \nabla^h  \ast_{h_t} (e^t_u\wedge e^t_l) =    \ast_{h_t} (\nabla^{h_t} e^t_u\wedge e^t_l) + \ast_{h_t} ( e^t_u\wedge \nabla^{h_t} e^t_l) = \Theta_t(e^t_n)\otimes e^t_u \, .
\end{equation*}

\noindent
The skew-symmetrization of the previous equation together with the  skew-symmetrization of equations \eqref{eq:globalesp2} yields the second equation in \eqref{eq:globhyperspinorflowI}. Conversely, suppose that $\left\{ \fre^t , \beta_t\right\}_{t\in \cI}$ is a solution of equations \eqref{eq:globhyperspinorflowI} and \eqref{eq:globhyperspinorflowII}, and set:
\begin{equation*}
	h_{\fre^t} = e_u^t\otimes e_u^t + e_l^t\otimes e_l^t + e_n^t\otimes e_n^t\, , \qquad \Theta_t = - \frac{1}{2\beta_t}\partial_t h_{\fre^t}\, .
\end{equation*}

\noindent
Since $[\Theta_t(e^t_u )]=0$ in $H^1(\Sigma,\mathbb{R}) = 0$, there exists a smooth family of functions $\left\{ \bar{\frf}_t \right\}_{t\in \mathbb{R}}$ such that:
\begin{equation*}
\dd\bar{\frf}_t = - \Theta_t(e^t_u ) \, . 
\end{equation*}

\noindent
Taking the time-derivative of the previous expression we obtain:
\begin{equation*}
\dd\partial_t\bar{\frf}_t = - \partial_t(\Theta_t(e^t_u)) \, .
\end{equation*}

\noindent
Hence, comparing with the first equation in \eqref{eq:globhyperspinorflowII} we conclude:
\begin{equation*}
\dd\partial_t\bar{\frf}_t = \dd (\dd\beta_t(e^t_u))\, .
\end{equation*}

\noindent
implying $\partial_t \bar{\frf}_t = \dd\beta_t(e^t_u) + c(t)$ for a certain function $c(t)$ depending exclusively on $t$. Set $\frf_t := \bar{\frf}_t - \int c(\tau) \dd \tau$. By construction we have $\partial_t \frf_t = \dd\beta_t(e^t_u)$. Furthermore:
\begin{equation*}
\dd\partial_t\frf_t = - \partial_t\Theta_t(e^t_u ) \, .
\end{equation*}

\noindent
Define now $u^{\perp}_t := e^{\frf_t}\, e_u^t$ and $l^{\perp}_t := e^t_l$. The fact that both $e_u^t$ and $e_l^t$ satisfy the first equation in \eqref{eq:globhyperspinorflowI} implies:
\begin{equation*}
\partial_t u^{\perp}_t + \beta_t \Theta_t(u^{\perp}_t) + (\dd\beta_t(e^t_u) -\partial_t \frf_t) u^{\perp}_t = u^0_t  \dd \beta_t \, ,
\end{equation*}

\noindent
as well as:
\begin{equation*}
u^0_t \partial_t l^{\perp}_t +\beta_t  u^0_t\Theta_t(l^{\perp}_t) + \dd \beta_t (l^{\perp}_t) u^{\perp}_t = 0  \, ,
\end{equation*}

\noindent
where we have identified $u^0_t := e^{\frf_t}$. Using the fact that $\partial_t \frf_t = \dd\beta_t(e^t_u)$ we obtain equations \eqref{eq:globalesp1}. Equations \eqref{eq:globalesp2} follow directly by interpreting the second equation in \eqref{eq:globhyperspinorflowI} as the first Cartan structure equations for the coframe $\fre^t$, considered as orthonormal with respect to the metric $h_{\fre^t} = e^t_u\otimes e^t_u + e^t_l\otimes e^t_l + e^t_n\otimes e^t_n$. Finally, equations \eqref{eq:globalesp3} hold by construction and hence we conclude.
\end{proof}

\begin{definition}
Equations  \eqref{eq:globhyperspinorflowI} and \eqref{eq:globhyperspinorflowII} are the \emph{real parallel spinor flow equations}. A \emph{real parallel spinor flow} is a family $\left\{ \beta_t ,\fre^t\right\}_{t\in \cI}$ of functions and coframes on $\Sigma$ satisfying the real parallel spinor flow equations.
\end{definition}

\noindent
Therefore, a globally hyperbolic Lorentzian four-manifold admits a parallel spinor if and only if it admits a Cauchy surface carrying a parallel spinor flow $\left\{\beta_t , \fre^t\right\}_{t\in \cI}$. In particular, the corresponding parallel spinor $\varepsilon$ can be fully reconstructed from $\left\{ \beta_t , \fre^t\right\}_{t\in \cI}$. We remark that for our purposes the explicit expression of the parallel spinor associated to a given parallel spinor flow $\left\{ \beta_t , \fre^t\right\}_{t\in \cI}$ is of no relevance in itself. Instead, we are interested in the geometric and topological consequences associated to the existence of a parallel spinor $\varepsilon$, rather than on its specific expression.


\subsection{Admissible common initial data}
\label{eq:constraint}


The parallel spinor flow equations pose an evolution problem whose associated constraint equations are equivalent to the constraint equations of the evolution problem posed by a parallel spinor on a globally hyperbolic Lorentzian four-manifold.  Taking $\Sigma := \Sigma_0$ as the Cauchy hypersurface of $(M,g)$ and setting:
\begin{equation*}
\fre := \fre^0\, , \quad \Theta := \Theta_0\, ,
\end{equation*}

\noindent
the restriction of the second set of equations in \eqref{eq:globhyperspinorflowI} and of the second equation in \eqref{eq:globhyperspinorflowII} to $\Sigma$ reads:
\begin{eqnarray}
\label{eq:exderKI}
& \dd e_u=\Theta(e_u)\wedge e_u \, , \quad \dd e_l=\Theta(e_l) \wedge e_u  \, , \quad \dd e_n   = \Theta(e_n) \wedge e_u \, ,\\
& [\Theta(e_u) ] = 0 \in H^1(\Sigma,\mathbb{R})\, .\label{eq:exderKII}
\end{eqnarray}

\noindent
We will consider equations \eqref{eq:exderKI} and \eqref{eq:exderKII} as the constraint equations of the parallel spinor flow, whose solutions $(\fre,\Theta)$ are by definition the allowed initial data of the parallel spinor flow. We will refer to equations \eqref{eq:exderKI} and \eqref{eq:exderKII} as the \emph{parallel Cauchy differential system}.    
 
\begin{definition}
\label{def:KCC}
A \emph{parallel Cauchy pair} $(\fre,\Theta)$ is a solution of the parallel Cauchy differential system.  
\end{definition}

\noindent
We denote by $\Conf(\Sigma)$ the configuration space of the parallel Cauchy differential system, that is, its space of variables $(\fre,\Theta)$, whereas we denote by $\Sol(\Sigma)$ the space of parallel Cauchy pairs. Note that the function $\beta_0$ does not occur in the parallel Cauchy differential system, exactly as it happens with the initial value problem posed by the Ricci flat condition of a Lorentzian metric \cite{Yvonne,ChoquetBruhatBook}. Given a pair $(\fre,\Theta) \in \Conf(\Sigma)$, we denote by $h_{\fre}$ the Riemannian metric on $\Sigma$ defined by:
\begin{equation*}
h_{\fre} = e_{u}\otimes e_u + e_{l}\otimes e_l  + e_{n}\otimes e_n\, ,
\end{equation*}

\noindent
where $\fre = (e_u,e_l,e_n)$. We say that $(\fre,\Theta)$ is complete if $h_{\fre}$ is a complete Riemannian metric on $\Sigma$. Denote by $\mathrm{Met}(\Sigma) \times \Gamma(T^{\ast}\Sigma\odot T^{\ast}\Sigma)$ the set of pairs consisting of Riemannian metrics and symmetric two-tensors on $\Sigma$. We obtain a canonical map:
\begin{equation*}
\Psi\colon \Conf(\Sigma) \to \mathrm{Met}(\Sigma) \times \Gamma(T^{\ast}M\odot T^{\ast}M)\, , \qquad (\fre,\Theta) \mapsto (h_{\fre},\Theta)\, .
\end{equation*}

\noindent
The set $\mathrm{Met}(\Sigma) \times \Gamma(T^{\ast}M\odot T^{\ast}M)$ is in fact the configuration space of the constraint equations associated to the Cauchy problem posed by the Ricci flat condition on a globally hyperbolic Lorentzian four-manifold with Cauchy surface $\Sigma$, which are given by \cite{Yvonne,ChoquetBruhatBook}:
\begin{equation}
\label{eq:EinsteinonSigma}
\mathrm{R}_h = \vert\Theta\vert^2_h - \mathrm{Tr}_h(\Theta)^2\, ,\qquad  \dd \mathrm{Tr}_h(\Theta) = \mathrm{div}_h(\Theta)\, ,
\end{equation}

\noindent
for pairs $(h,\fre) \in \mathrm{Met}(\Sigma) \times \Gamma(T^{\ast}M\odot T^{\ast}M)$. 

\begin{remark}
The first equation in \eqref{eq:EinsteinonSigma} is usually called the \emph{Hamiltonian constraint} whereas the second equation in \eqref{eq:EinsteinonSigma} is usually called the \emph{momentum constraint}. 
\end{remark}

\noindent
Therefore, the map $\Psi$ provides a natural link between the initial value problem associated to a parallel spinor and the initial value problem associated to the Ricci-flatness condition. In particular, it allows introducing a natural notion of \emph{admissible} initial data to both evolution problems. 

\begin{definition}
A parallel Cauchy pair $(\fre,\Theta)$ is \emph{constrained Ricci flat} if $(h_{\fre},\Theta)$ satisfies the momentum and Hamiltonian constraints \eqref{eq:EinsteinonSigma}.
\end{definition}

\noindent
In references \cite{BaumLeistnerLischewski,LeistnerLischewski} it was proven that the Cauchy problem posed by a parallel spinor is well-posed, implying that every parallel Cauchy pair admits a Lorentzian development carrying a parallel spinor and hence a parallel spinor flow. However, since a Lorentzian metric admitting a parallel spinor is not necessarily flat, it might not be possible to \emph{evolve} a constrained Ricci flat parallel Cauchy pair $(\fre,\Theta)$ in such a way that both the Ricci-flatness condition and the existence of a parallel spinor are guaranteed. We will solve this problem in the affirmative in the next subsection and, in doing so, we will obtain an \emph{initial data characterization} of parallel spinors on globally hyperbolic Ricci flat four-manifolds.


\subsection{Initial data characterization}


Denote by $\cP(\Sigma)$ the set of parallel spinor flows on $\Sigma$, that is, the set of families $\left\{\beta_t,\fre^t\right\}_{t\in\cI}$ satisfying the parallel spinor flow equations \eqref{eq:globhyperspinorflowI} and \eqref{eq:globhyperspinorflowII}. We have a canonical map:
\begin{equation*}
\Phi \colon \cP(\Sigma) \to \mathrm{Lor}_{\circ}(M)\, , \qquad \left\{\beta_t,\fre^t\right\}_{t\in\cI} \mapsto  g= -\beta_t^2 \dd t^2 + h_{\fre^t}\, ,
\end{equation*}

\noindent
from $\cP(\Sigma)$ to the set $\mathrm{Lor}_{\circ}(M)$ of globally hyperbolic Lorentzian metrics on $M=\cI\times \Sigma$. For simplicity in the exposition, we will refer to $\Phi(\left\{\beta_t,\fre^t\right\})$ as the globally hyperbolic metric determined by $\left\{\beta_t,\fre^t\right\}_{t\in\cI}$. Given a parallel spinor flow $\left\{\beta_t,\fre^t\right\}_{t\in\cI}$, there exists a smooth family functions $\left\{ \frf_t\right\}_{t\in\cI}$ such that:
\begin{equation*}
\dd\frf_t = - \Theta_t(e^t_u)\, , \qquad \partial_t \frf_t  = \dd \beta_t (e^t_u)\, , 
\end{equation*}

\noindent
which is unique modulo the addition of a real constant. Using this family of functions, we obtain a canonical map:
\begin{equation*}
\Xi \colon \cP(\Sigma) \to \cB(M)\, , \qquad \left\{\beta_t,\fre^t\right\}_{t\in\cI} \mapsto  (u =  e^{\frf_t} (\beta_t \dd t + e^t_u), [l=e^t_l])\, ,
\end{equation*}

\noindent
from the set of parallel spinor flows on $\Sigma$ to the set $\cB(M)$ of parabolic pairs on $M$ with respect to the globally hyperbolic metric defined by the given parallel spinor flow. The previous maps provide a construction which is essentially \emph{inverse} to the splitting and reduction implemented at the beginning of Section \ref{sec:GlobhyperSpinors4d} and which allows us to relate properties of a given parallel spinor flow to properties of its associated globally hyperbolic  four-dimensional Lorentzian metric. For further reference, we introduce the \emph{Hamiltonian function} of a parallel spinor flow $\left\{\beta_t,\fre^t\right\}_{t\in\cI}$ as follows:
\begin{equation*}
\cH \colon M = \mathbb{R}\times \Sigma \to \mathbb{R}\, , \qquad (t,p) \mapsto (\mathrm{R}_{h_t}- \vert \Theta_t \vert_{h_t}^2+\mathrm{Tr}_{h_t} (\Theta_t)^2)\vert_p\, ,
\end{equation*}

\noindent
where $h_t := h_{\fre^t}$ denotes the three-dimensional metric restricted to the Cauchy surface $\Sigma_t$ and $\mathrm{R}_{h_{\fre^t}}$ its scalar curvature. 

\begin{proposition}
\label{prop:ricci}
Let $\left\{\beta_t,\fre^t\right\}_{t\in\cI}$ be a parallel spinor flow on $\Sigma$. The Ricci curvature of $g = \Phi(\left\{\beta_t,\fre^t\right\})$ reads:
\begin{equation}
\label{eq:ricci4d}
 \ric^g = \frac{1}{2} \mathcal{H} e^{- 2 \frf_t}  u\otimes u\, ,
\end{equation}

\noindent
where $\Xi(\left\{\beta_t,\fre^t\right\}) = (u,[l])$ and $u = e^{\frf_t}(\beta_t\dd t + e^t_u)$.
\end{proposition}

\begin{proof}
Let $\left\{\beta_t,\fre^t\right\}_{t\in\cI}$ be a parallel spinor flow on $\Sigma$ and let $g = \Phi(\left\{\beta_t,\fre^t\right\}) = -\beta^2_t \dd t \otimes \dd t + h_t$ its associated globally hyperbolic metric on $M = \mathbb{R}\times \Sigma$. The pair $\left\{\beta_t,\fre^t\right\}_{t\in\cI}$ defines a global orthonormal coframe $(e_0,e_1,e_2,e_3)$ on $(M,g)$ given by:
\begin{equation*}
	e_0\vert_{(t,p)} : = \beta_t\vert_{p}  \dd t \, , \qquad e_1\vert_{(t,p)}  := e^t_u\vert_{p} \, , \qquad e_2\vert_{(t,p)}  := e^t_l\vert_{p} \, , \qquad e_3\vert_{(t,p)}  := e^t_n\vert_{p} \, .
\end{equation*}

\noindent
The fact that $\left\{\beta_t,\fre^t\right\}_{t\in\cI}$ is a parallel spinor flow implies that the exterior derivatives of the coframe $(e_0,e_1,e_2,e_3)$ on $M$ are prescribed as follows:
\begin{equation*}
\dd e_0 = \dd\log(\beta_t) \wedge e_0 \, , \quad \dd e_a = (\dd\log(\beta_t)(e_a)\, e_1 +  \Theta_t(e_a) - \delta_{a1} \dd \log(\beta_t)) \wedge e_0 + \Theta_t (e_a) \wedge e_1\, ,
\end{equation*}

\noindent
where $a = 1,2,3$. Interpreting the previous expression as the first Cartan structure equations for $\nabla^g$ with respect to the orthonormal coframe $(e_0,e_1,e_2,e_3)$ and using repeatedly Equations \eqref{eq:globhyperspinorflowI} and \eqref{eq:globhyperspinorflowII}, a tedious calculation yields Equation \eqref{eq:ricci4d} and hence we conclude.
\end{proof}

\begin{remark}
It is well-known that the Ricci curvature $\mathrm{Ric}^g$ of a Lorentzian four-manifold admitting parallel spinors is of the form $\mathrm{Ric}^g = f u\otimes u$ for some function $f\in C^{\infty}(M)$ \cite{Mars}. Nonetheless, and to the best of our knowledge, Equation \eqref{eq:ricci4d} is the first precise characterization of such function $f$ in the case of globally hyperbolic Lorentzian four-manifolds.  
\end{remark}
 
\begin{theorem}
\label{thm:preservingconstraints}
The parallel spinor flow preserves the vacuum momentum and Hamiltonian constraints.
\end{theorem}

\begin{proof}
Let $\left\{\beta_t,\fre^t\right\}_{t\in\cI}$ be a parallel spinor flow. Taking the divergence of Equation \eqref{eq:ricci4d} we obtain:
\begin{equation*}
\dd(\cH e^{-2\frf_t}) (u^{\sharp}) = 0\, ,
\end{equation*}

\noindent
which can be equivalently written as follows:
\begin{equation*}
\mathfrak{D}(\cH) = \rho \cH\, ,
\end{equation*}

\noindent
where $\mathfrak{D} = \partial_t + e^t_u$ is a first-order symmetric hyperbolic differential operator and $\rho = 2 (\partial_t \frf_t + e^t_u(\frf_t))$ is a function completely determined by $\left\{\beta_t,\fre^t\right\}_{t\in\cI}$. The fact that we are able to rewrite equation $\dd(\cH e^{-2\frf_t}) (u^{\sharp}) = 0$ in terms of the symmetric hyperbolic differential operator $\mathfrak{D}$ follows from the fact that $u$ is nowhere vanishing and null, which in turn implies that the component $u_t$ can never vanish on $\cI\times \Sigma$ and is therefore of constant sign. Given such $\mathfrak{D}$ and $\rho$ associated to $\left\{\beta_t,\fre^t\right\}_{t\in\cI}$, consider now the initial value problem:
\begin{equation*}
\mathfrak{D}(F) = \rho\, F\, , \qquad F\vert_{\Sigma} = 0\, ,
\end{equation*}

\noindent
for an arbitrary function $F$ on $M$. By the existence and uniqueness theorem for this type of equations, see \cite[Theorem 19]{CortesKronckeLouis}, every solution must be zero on a neighborhood of $\Sigma$. Since $\cH$ is in particular a solution of this equation, it must vanish on a neighborhood of $\Sigma$. Therefore, there exists a subinterval $\cI^{\prime} = (a,b) \subseteq \cI$ containing zero such that $\cH\vert_t = 0$ for every $t\in \cI^{\prime}$. By \cite[Proposition 2.24]{Murcia:2020zig} this implies that the momentum constraint is also satisfied for every $t\in \cI^{\prime}$ and hence the parallel spinor flow preserves the Hamiltonian and momentum constraints on $\cI^{\prime}$. If $\cI^{\prime} = \cI$  we conclude, so assume that $\cI^{\prime} = (a,b) \subset \cI$ is the proper maximal subinterval of $\cI$ for which the result holds. Since the parallel spinor flow $\{\beta_t, \fre^t\}_{t \in \cI}$ is well-defined in $\cI$, then both $\rho$ and $\mathcal{H}$ must be well-defined on $\cI\times \Sigma$. Hence, by point-wise continuity on $\Sigma$ we must have that $\cH\vert_{b}=0$ and therefore we can apply the previous argument to the initial value problem \emph{starting} at $b \in \cI$. Hence there exists an $\varepsilon > 0$ for which the result holds on $(a, b+\varepsilon)$, in contradiction with $(a,b)$ being maximal. Therefore, $\cI^{\prime} = \cI$ and we conclude.
\end{proof}

\noindent
Call a triple $(\Sigma,h,\Theta)$ an \emph{initial vacuum} data if $(h,\Theta)$ satisfies the Hamiltonian and momentum constraints. The previous theorem can be applied to prove an \emph{initial data} characterization of parallel spinors on Ricci flat Lorentzian four-manifolds. 

\begin{corollary}
An initial vacuum data $(\Sigma,h,\Theta)$ admits a Ricci flat Lorentzian development carrying a parallel spinor if and only if there exists a global orthonormal coframe $\fre$ on $\Sigma$ such that $(\fre,\Theta)$ is a parallel Cauchy pair.
\end{corollary}

\begin{proof}
The \emph{only if} condition follows from Theorem \ref{thm:parallelspinorflow}. For the \emph{if} condition, let $(\Sigma,h,\Theta)$ be an initial vacuum data. If in addition there exists a global orthonormal coframe $\fre$ on $\Sigma$ such that $(\fre,\Theta)$ is a parallel Cauchy pair, then the constraint equations of the initial value problem of a parallel spinor are satisfied. By \cite{LeistnerLischewski,Lischewski:2015cya}, the initial value problem is well-posed and there exists a Lorentzian development of $\Sigma$ carrying a parallel spinor. By Theorem \ref{thm:preservingconstraints}, this Lorentzian development satisfies the Hamiltonian and momentum constraint for every $t\in \cI$, and by Equation \eqref{eq:ricci4d} we conclude that this Lorentzian development is Ricci flat.
\end{proof}

\noindent
Additionally, we obtain the following corollary.

\begin{corollary}
A globally hyperbolic Lorentzian four-manifold $(M,g)$ admitting a parallel spinor is Ricci flat if and only if there exists a Cauchy hypersurface $\Sigma\subset M$ whose Hamiltonian constraint vanishes.
\end{corollary}


\subsection{Left-invariant parallel Cauchy pairs}
 

Let $\Sigma = \G$ be a simply connected three-dimensional Lie group. A pair $(\fre,\Theta) \in \Conf(\G)$ is said to be \emph{left-invariant} if both $\fre$ and $\Theta$ are left-invariant. Given a left-invariant pair $(\fre,\Theta) \in \Conf(\G)$ we write:
\begin{equation*}
\Theta = \Theta_{ab}\, e_a\otimes e_b\, , \quad a, b = u, l, n\, ,
\end{equation*}

\noindent
where summation over repeated indices is understood. For further reference we introduce the following notation:
\begin{equation*}
	\lambda := \sqrt{\Theta_{ul}^2 + \Theta_{un}^2}\, , \quad \theta := \begin{pmatrix}
		\tll & \tln \\
		\tln & \tnn
	\end{pmatrix} \, , \quad T := \mathrm{Tr}(\theta)\, ,  \quad \Delta := \mathrm{Det}(\theta) = \Theta_{ll}\Theta_{nn} - \Theta_{ln}^2\, .
\end{equation*}

\noindent
These will play an important role in the classification of left-invariant Cauchy pairs, which was completed in \cite{Murcia:2020zig} and which we proceed to summarize.

\begin{theorem}\cite[Theorem 4.9]{Murcia:2020zig}
\label{thm:allcauchygroups}
A connected and simply-connected Lie group $\G$ admits left-invariant parallel Cauchy pairs (respectively constrained Ricci flat parallel Cauchy pairs) if and only if $\G$ is isomorphic to one of the Lie groups listed in the table below. If that is the case, a left-invariant shape operator $\Theta$ belongs to a Cauchy pair $(\fre,\Theta)$ for certain left-invariant coframe $\fre$ if and only if $\Theta$ is of the form listed below when written in terms of $\fre = (e_u,e_l,e_n)$:
\renewcommand{\arraystretch}{1.5}
\begin{center}
\begin{tabular}{|  p{1cm}| p{8.8cm} | p{3.8cm} |}
\hline
$\mathrm{G}$ & \emph{Cauchy parallel pair} & \emph{Constrained Ricci flat}  \\ \hline
$\mathbb{R}^3$ & $\Theta=\Theta_{uu} e_u \otimes e_u$ &  $\Theta=\Theta_{uu} e_u \otimes e_u$   \\ \hline
\multirow{2}*{$ \mathrm{E}(1,1)$} & $\Theta=\Theta_{uu} e_u \otimes e_u+ \Theta_{i j} e_i \otimes e_j$ & \multirow{2}*{\emph{Not allowed}} \\ & $i,j=l,n,\, \quad \Theta_{ll}=-\Theta_{nn}$ &  \\ 
\hline \multirow{11}*{$\tau_2 \oplus \mathbb{R}$} & $\Theta=(\tul e_l +\tun e_n) \odot e_u$  & \multirow{2}*{\emph{Not allowed}} \\ & $\tul^2+\tun^2 \neq 0$ &   \\ \cline{2-3} & $\Theta=\Theta_{uu} e_u \otimes e_u+ \Theta_{i j} e_i \otimes e_j$ &  $\Theta=T e_u \otimes e_u+ \Theta_{i j} e_i \otimes e_j$  \\ & $\begin{aligned} &i,j=l,n,\, \\ &T \neq 0\, , \Delta=0 \end{aligned}$ & $\begin{aligned} &i,j=l,n,\, \\ &T \neq 0\, , \Delta=0 \end{aligned}$\\ \cline{2-3} & $\Theta=-T e_u \otimes e_u+\tul e_u \odot e_l+\tll e_l \otimes e_l\, , \quad \tul, \tll \neq 0$ & \emph{Not allowed}  \\ \cline{2-3} & $\Theta=-T e_u \otimes e_u+\tun e_u \odot e_n+\tnn e_n \otimes e_n\, , \quad \tun, \tnn \neq 0$ & \emph{Not allowed}  \\\cline{2-3} & $\Theta=-T e_u\otimes e_u + \tul e_u \odot e_l+ \tun e_u \odot e_n+\Theta_{i j} e_i \otimes e_j\,$ & \multirow{3}*{\emph{Not allowed}} \\  & $\begin{aligned}&i,j=l,n,\, \, \tln\tul\tun \neq 0\, , \\&\tnn=\frac{\tun}{\tul} \tln\, , \tll=\frac{\tul}{\tun} \tln \end{aligned}$ &  \\ & &   \\ \hline\multirow{2}*{$ \tau_{3,\mu}$} & $\Theta=\Theta_{uu} e_u \otimes e_u+ \Theta_{i j} e_i \otimes e_j$ & $\Theta=\left (\frac{T^2-2\Delta}{T} \right ) e_u \otimes e_u+ \Theta_{i j} e_i \otimes e_j$   \\ & $i,j=l,n,\, \quad T, \Delta \neq 0$& $i,j=l,n,\, \quad T, \Delta \neq 0$  \\ \hline
\end{tabular}		
\end{center}
Regarding the case $\G =\tau_{3,\mu}$:
\begin{itemize}[leftmargin=*]
\item If $\tln \neq 0$, then
\begin{equation*}
\mu = \frac{T-\text{\emph{sign}}(T)\sqrt{T^2-4\Delta}}{T+\text{\emph{sign}}(T)\sqrt{T^2-4\Delta}}\, .
\end{equation*}
\item If $\tln=0$ and $\vert\tll\vert \geq \vert \tnn \vert $, then
\begin{equation*}
\mu=\frac{\tnn}{\tll}\,.
\end{equation*}
\item If $\tln=0$ and $\vert \tnn \vert \geq \vert \tll \vert $, then
\begin{equation*}
\mu=\frac{\tll}{\tnn}\,.
\end{equation*}
\end{itemize}
\renewcommand{\arraystretch}{1}
\end{theorem}

\noindent
The previous theorem will be used extensively in the next section. We have the following corollary.

\begin{corollary}
\label{cor:isog}
Let $\G$ be a connected and simply connected Lie group with a left-invariant Cauchy pair. Then the isomorphism type of $\G$ is prescribed by $T$, $\Delta$ and $\lambda$ as follows:
\begin{itemize}
\item If $T=\Delta=\lambda=0$, then $\G \simeq \mathbb{R}^3$.
\item If $T=\lambda=0$ but $\Delta \neq 0$, then $\G \simeq \mathrm{E}(1,1)$.
\item If  $\Delta = 0$ but $\lambda^2 + T^2 \neq 0$, then $\G \simeq \tau_2\oplus \mathbb{R}$.
\item If $T, \Delta \neq 0$ and $\lambda=0$, then $\G \simeq \tau_{3,\mu}$.
\end{itemize}

\noindent
Observe that the case $\lambda \neq 0$ and $\Delta \neq 0$ is not allowed. 
\end{corollary}

\noindent
We are using standard notation for the groups $\G$ as explained for example in \cite[Appendix A]{Freibert}.
 

\section{Left-invariant parallel spinor flows}
\label{sec:leftinvariantspinorflow}


In this section we introduce the notion of left-invariant parallel spinor flow and solve it explicitly. 


\subsection{Reformulation }


Let $\G$ be a simply connected three-dimensional Lie-group. We say that a parallel spinor flow $\left\{ \beta_t ,\fre^t \right\}_{t\in \cI}$ defined on $\G$ is \emph{left-invariant} if both $\beta_t$ and $\fre^t$ are left-invariant for every $t\in \cI$. The latter condition immediately implies that $h_{\fre^t}$ is a left-invariant Riemannian metric and $\beta_t$ is constant for every $t\in \cI$. Let $\left\{ \fre^t\right\}_{t\in\cI}$ be a family of left-invariant coframes on $\G$. Any square matrix $\cA\in \mathrm{Mat}(3,\mathbb{R})$ acts naturally on $\left\{ \fre^t\right\}_{t\in \cI}$ as follows: 
\begin{equation*}
\cA(\fre^t) := 
\begin{pmatrix}
\sum_b \cA_{ub} e^t_b  \\
\sum_b \cA_{lb} e^t_b  \\
\sum_b \cA_{nb} e^t_b
\end{pmatrix}
\end{equation*}

\noindent
where we label the entries $\cA_{ab}$ of $\cA$ by the indices $a,b = u, l, n$. As a direct consequence of Theorem \ref{thm:parallelspinorflow} we have the following result.

\begin{proposition}
A simply connected three-dimensional Lie group  $\G$ admits a left-invariant parallel spinor flow if and only if there exists a smooth family of non-zero constants $\left\{ \beta_t\right\}_{t\in \cI}$ and a family $\left\{ \fre^t \right\}_{t\in \cI}$ of left-invariant coframes on $\G$ satisfying the following differential system:
\begin{eqnarray}
\label{eq:leftinv}
\partial_t \fre^t  +  \beta_t\Theta_t(\fre^t) = 0\, , \quad \dd \fre^t  = \Theta_t(\fre^t) \wedge e^t_u\, , \quad \partial_t(\Theta_t(e^t_u )) = 0\, , \quad \dd\Theta_t(e^t_u ) = 0\, ,
\end{eqnarray}
	
\noindent
to which we will refer as the left-invariant (real) parallel spinor flow equations.
\end{proposition}

\noindent
We will refer to solutions $\left\{ \beta_t, \fre^t \right\}_{t\in \cI}$ of the left-invariant parallel spinor flow equations as \emph{left-invariant parallel spinor flows}. Given a parallel spinor flow $\left\{ \beta_t, \fre^t \right\}_{t\in \cI}$, we write:
\begin{equation*}
\Theta^t = \sum_{a,b} \Theta^t_{ab} e^t_a\otimes e^t_n\, , \qquad a, b = u, l , n\, ,
\end{equation*}

\noindent
in terms of uniquely defined functions $(\Theta^t_{ab})$ on $\cI$.

\begin{lemma}
\label{lemma:intecondli}
Let $\{\beta_t, \fre^t\}_{t \in \mathcal{I}}$ be a left-invariant parallel spinor flow. The following equations hold:
\begin{equation*}
\begin{split}
\partial_t \tuu^t=\beta_t ((\tuu^t)^2 + (\tul^t )^2+(\tun^t)^2)\,,& \quad \partial_t \tul^t=\partial_t \tun^t=0\, , \quad \partial_t \tll^t=\beta_t\tll^t \tuu^t-\beta_t (\tul^t)^2\, \\ 
\partial_t \tln^t= \beta_t  \tln^t\tuu^t-\beta_t \tun^t \tul^t\,, & \quad \partial_t \tnn^t=\beta_t\tnn^t \tuu^t-\beta_t
(\tun^t)^2\, , \\
\tln^t \tul^t=\tll^t \tun^t\, ,& \quad \tln^t \tun^t=\tnn^t \tul^t\, , \\ 
\tll^t \tul^t+\tln^t\tun^t+\tul^t\tuu^t=0\,, &\quad \tln^t \tul^t+\tnn^t \tun^t+\tun^t \tuu^t=0\, .
\end{split}
\end{equation*}

\noindent
In particular, $\tul^t=\tul$ and $\tun^t=\tun$ for some constants $\tul,\tun \in \mathbb{R}$.
\end{lemma}

\begin{proof}
A direct computation shows that equation $\partial_t(\Theta_t(e^t_u )) = 0$ is equivalent to:
\begin{equation*}
\partial_t \Theta^t_{ub} = \beta_t \Theta^t_{ua} \Theta^t_{ab}\, .
\end{equation*}

\noindent
On the other hand, equation $\dd \Theta_t(e^t_u) = 0$ is equivalent to:
\begin{equation*}
\Theta^t_{ua} \Theta^t_{al} = 0\, , \qquad \Theta^t_{ua} \Theta^t_{an} = 0\, .
\end{equation*}

\noindent
The previous equations can be combined into the following equivalent conditions:
\begin{eqnarray*}
& \partial_t \tuu^t=\beta_t ((\tuu^t)^2 + (\tul^t )^2+(\tun^t)^2)\, , \quad \partial_t \tul^t=\partial_t \tun^t=0\, ,\\
& \tll^t \tul^t+\tln^t\tun^t+\tul^t\tuu^t=0\, , \quad \tln^t \tul^t+\tnn^t \tun^t+\tun^t \tuu^t=0\, ,
\end{eqnarray*}

\noindent
which recover five of the equations in the statement. Similarly, equation $\dd (\Theta_t(\fre^t) \wedge e^t_u) = 0$ is equivalent to:
\begin{equation*}
\tln^t \tul^t=\tll^t \tun^t\, ,\qquad \tln^t \tun^t=\tnn^t \tul^t\, ,
\end{equation*}

\noindent
which yields the third line of equations in the statement. We take now the exterior derivative of the first equation in \eqref{eq:leftinv} and combine the result with the second equation in \eqref{eq:leftinv}:
\begin{equation*}
\dd (\partial_t e^t_a  +  \beta_t\Theta_t(e^t_a) ) =  \partial_t (\Theta^t_{ab} e^t_b\wedge e^t_u) + \beta_t \Theta^t_{ab} \Theta^t_{bc} e^t_c\wedge e^t_u  =  (\partial_t \Theta^t_{ab} \delta_{uc} - \beta_t \Theta^t_{ab} \Theta^t_{uc}) e^t_b\wedge e^t_c = 0\, .
\end{equation*}
\noindent
Expanding the previous equation we obtain the remaining three equations in the statement and we conclude.
\end{proof}

\begin{remark}
We will refer to the equations of Lemma \ref{lemma:intecondli} as the \emph{integrability conditions} of the left-invariant parallel spinor flow.
\end{remark}

\noindent
The following observation is crucial in order to \emph{decouple} the left-invariant parallel spinor flow equations.

\begin{lemma}
\label{lemma:cKTheta}
A pair $\left\{ \beta_t, \fre^t \right\}_{t\in \cI}$ is a left-invariant parallel spinor flow if and only if there exists  a family of left-invariant two-tensors $\left\{ \cK_t \right\}_{t\in \cI}$ such that the following equations are satisfied:
\begin{eqnarray*}
\partial_t \fre^t  +  \beta_t\cK_t(\fre^t) = 0\, , \quad \dd \fre^t  = \cK_t(\fre^t) \wedge e^t_u\, , \quad \partial_t(\cK_t(e^t_u )) = 0\, , \quad \dd(\cK_t(e^t_u)) = 0\, .
\end{eqnarray*}
\end{lemma}

\begin{proof}
The \emph{only if} direction follows immediately from the definition of left-invariant parallel Cauchy pair by taking $\{\cK_t\}_{t \in \cI}=\{\Theta_t\}_{t \in \cI}$. For the \emph{if} direction we simply compute:
\begin{equation*}
\Theta_t = - \frac{1}{2\beta_t}\partial_t h_{\fre^t} = - \frac{1}{2\beta_t} (( \partial_t e^t_a)\otimes e^t_a + e^t_a \otimes ( \partial_t e^t_a)) = \cK_t\, ,
\end{equation*}

\noindent
hence equations \eqref{eq:leftinv} are satisfied and $\left\{ \beta_t, \fre^t \right\}_{t\in \cI}$ is a left-invariant parallel spinor flow.
\end{proof}

\noindent
By the previous Lemma we promote the components of $\left\{\Theta_t \right\}_{t\in \cI}$ with respect to the basis $\left\{\fre^t \right\}_{t\in \cI}$ to be independent variables of the left-invariant parallel spinor flow equations \eqref{eq:leftinv}. Within this interpretation, the variables of left-invariant parallel spinor flow equations consist of triples $\left\{\beta_t,\fre^t,\Theta^t_{ab} \right\}_{t\in \cI}$, where $\left\{\Theta^t_{ab} \right\}_{t\in \cI}$ is a family of symmetric matrices. On the other hand, the integrability conditions of Lemma \ref{lemma:intecondli} are interpreted as a system of equations for a pair $\left\{\beta_t,\Theta^t_{ab} \right\}_{t\in \cI}$. In particular, the first equation in \eqref{eq:leftinv} is linear in the variable $\fre^t$ and can be conveniently rewritten as follows. For any family of coframes $\left\{\fre^t\right\}_{t\in\cI}$, set $\fre = \fre^0$ and consider the unique smooth path:
\begin{equation*}
\U^t \colon \cI \to \Gl_{+}(3,\mathbb{R})\, , \qquad t\mapsto \U^t\, ,
\end{equation*}

\noindent
such that $\fre^t = \U^t(\fre)$, where  $\Gl_{+}(3,\mathbb{R})$ denotes the identity component in the general linear group $\Gl(3,\mathbb{R})$. More explicitly:
\begin{equation*}
\fre^t_a = \sum_b \U^t_{ab} \fre_b\, , \qquad a, b = u, l ,n\, ,
\end{equation*}

\noindent
where $\U^t_{ab} \in C^{\infty}(\G)$ are the components of $\U^t$. Plugging $\fre^t = \U^t(\fre)$ in the first equation in \eqref{eq:leftinv} we obtain the following equivalent equation:
\begin{equation}
\label{eq:ptu}
\partial_t \U^t_{ac} + \beta_t \Theta^t_{ab} \U^t_{bc} = 0\, ,  \quad a, b, c = u, l, n\, ,
\end{equation}

\noindent
with initial condition $\U^0 = \mathrm{Id}$. 

\noindent 
A necessary condition for a solution $\left\{\beta_t,\Theta^t_{ab} \right\}_{t\in \cI}$ of the integrability conditions to arise from an honest left-invariant parallel spinor pair is the existence of a left-invariant coframe $\fre$ on $\Sigma$ such that $(\fre,\Theta)$ is a Cauchy pair, where $\Theta = \Theta^0_{ab} e_a\otimes e_b$. Consequently we define the set $\mathbb{I}(\Sigma)$ of \emph{admissible} solutions to the integrability equations as the set of pairs $(\left\{\beta_t,\Theta^t_{ab} \right\}_{t \in \cI} , \fre )$ such that $\left\{\beta_t,\Theta^t_{ab} \right\}_{t\in \cI}$ is a solution to the integrability equations and $(\fre,\Theta)$ is a left-invariant parallel Cauchy pair. 

\begin{proposition}
\label{prop:bijectionsolutions}
There exists a natural bijection $\varphi\colon \mathbb{I}(\Sigma) \to \cP(\Sigma)$ which maps every pair:
\begin{equation*}
(\left\{\beta_t,\Theta^t_{ab} \right\}_{t\in\cI} , \fre )\in \mathbb{I}(\Sigma)\, ,
\end{equation*}

\noindent
to the pair $\left\{ \beta_t,\fre^t = \U^t(\fre)\right\}_{t\in\cI}\in \cP(\Sigma)$, where $\left\{\U^t\right\}_{t\in\cI}$ is the unique solution of \eqref{eq:ptu} with initial condition $\U^0 = \mathrm{Id}$. 
\end{proposition}

\begin{remark}
\label{remark:inverse}
The inverse of $\varphi$  maps every left-invariant parallel spinor flow $\left\{ \beta_t,\fre^t\right\}_{t\in\cI}$ to the pair $(\left\{\beta_t,\Theta^t_{ab} \right\} , \fre )$, where $\Theta^t_{ab}$ are the components of the shape operator associated to $\left\{ \beta_t,\fre^t\right\}_{t\in\cI}$ in the basis $\left\{\fre^t\right\}_{t\in\cI}$ and $\fre = \fre^0$.
\end{remark}

\begin{proof}
Let $(\left\{\beta_t,\Theta^t_{ab} \right\} , \fre )\in \mathbb{I}(\Sigma)$ and let  $\left\{\U^t\right\}_{t\in\cI}$ be the solution of \eqref{eq:ptu} with initial condition $\U^0 = \mathrm{Id}$, which exists and is unique on $\cI$ by standard ODE theory \cite[Theorem 5.2]{CoddingtonLevinson}. We need to prove that $\left\{ \beta_t,\fre^t = \U^t(\fre)\right\}_{t\in\cI}$ is a left-invariant parallel spinor flow. Since $\left\{\U^t\right\}_{t\in\cI}$  satisfies \eqref{eq:ptu} for the given $\left\{\beta_t,\Theta^t_{ab} \right\}$, it follows that $\Theta^t = \Theta^t_{ab} e^t_a\otimes e^t_b$ is the shape operator associated to $\left\{ \beta_t,\fre^t \right\}_{t\in\cI}$ whence the first equation in \eqref{eq:leftinv} is satisfied. On the other hand, the third and fourth equations in \eqref{eq:leftinv} are immediately implied by the integrability conditions satisfied by  $\left\{\beta_t,\Theta^t_{ab} \right\}$. Regarding the second equation in \eqref{eq:leftinv}, we observe that the integrability conditions contain the equation $\dd (\Theta_t(\fre^t) \wedge e_u^t)=0$ and thus:
\begin{equation}
\label{eq:constraintintegrada}
\dd \fre^t= \Theta^t(\fre^t) \wedge e_u^t+ \mathfrak{w}^t\, ,
\end{equation}

\noindent
where $\{ \mathfrak{w}^t\}_{t \in \cI}$ is a family of triplets of closed two-forms on $\Sigma$. Taking the time derivative of the previous equations, plugging the exterior derivative of the first equation in \eqref{eq:leftinv} and using again the integrability conditions, we obtain that $\mathfrak{w}^t$ satisfies the following differential equation:
\begin{equation}
\label{eq:frwode}
\partial_t \mathfrak{w}^t_a = -\beta_t \Theta^t_{ad} \mathfrak{w}^t_d\, ,
\end{equation}

\noindent
with initial condition $\mathfrak{w}^0 = \mathfrak{w}$. Restricting equation \eqref{eq:constraintintegrada} to $t=0$ it follows that $\mathfrak{w}$ satisfies:

\begin{equation*}
\dd \fre= \Theta(\fre) \wedge e_u +  \mathfrak{w}\, ,
\end{equation*}

\noindent
Since by assumption $(\fre,\Theta)$ is left-invariant Cauchy pair, the previous equation is satisfied if and only if $\mathfrak{w} =0$ whence $\mathfrak{w}^t = 0$ by uniqueness of solutions of the linear differential equation \eqref{eq:frwode}. Therefore, the second equation in \eqref{eq:leftinv}  follows and $\varphi$ is well-defined. The fact that $\varphi$ is in addition a bijection follows directly by Remark \ref{remark:inverse} and hence we conclude. 
\end{proof}
 
\begin{corollary}
\label{cor:solutionsiff}
A pair $\left\{\beta_t,\fre^t \right\}_{t\in \cI}$ is a parallel spinor flow if and only if $(\left\{\beta_t,\Theta^t_{ab} \right\} , \fre )$ is an admissible solution to the integrability equations.
\end{corollary}
 
\noindent
Therefore, solving the left-invariant parallel spinor flow is equivalent to solving the integrability conditions with initial condition $\Theta_{ab}$ being part of a left-invariant parallel Cauchy pair $(\fre,\Theta)$. We remark that $\left\{\beta_t \right\}_{t\in \cI}$ is of no relevance locally since it can be eliminated through a reparametrization of time after possibly shrinking $\cI$. However, regarding the long time existence of the flow as well as for applications to the construction of four-dimensional Lorentzian metrics it is convenient to keep track of $\cI$, whence we maintain $\left\{\beta_t \right\}_{t\in \cI}$ in the equations.

For further reference we define a \emph{quasi-diagonal}  left-invariant parallel spinor flow as one for which $\lambda=\sqrt{\tul^2+\tun^2} =0$. Since the function $t\to \int_{0}^t \beta_\tau \dd \tau$ is going to be a common occurrence in the following, we define:
\begin{equation*}
\cB_{t} := \int_{0}^t \beta_\tau \dd \tau\, .
\end{equation*}

\noindent
We distinguish now between the cases $\lambda = 0$ and $\lambda \neq 0$.

\begin{lemma}
\label{lemma:Thetaqdiagonal}
Let $\{\beta_t, \fre^t\}_{t \in \cI}$ be a quasi-diagonal left-invariant parallel spinor flow. Then, the only non-zero components of $\Theta^t$ are:
\begin{eqnarray*}
\tuu^t =\frac{\tuu}{1- \tuu \cB_t}\, , \quad \tll^t=\frac{\tll}{1- \tuu \cB_t}\, , \quad\tln^t =\frac{\tln}{1- \tuu \cB_t}\, , \quad \tnn^t=\frac{\tnn}{1- \tuu \cB_t}\, ,
\end{eqnarray*}

\noindent
where $\Theta^t$ is the shape operator associated to $\{\beta_t, \fre^t\}_{t \in \cI}$  and $\Theta = \Theta^0$. Furthermore, every such $\Theta^t$ satisfies the integrability equations with quasi-diagonal initial data.
\end{lemma}

\begin{proof}
Setting $\tul =\tun =0$ in the integrability conditions we obtain the following equations:
\begin{equation*}
\partial_t \tuu^t=\beta_t (\tuu^t)^2\, , \quad \partial_t \tll^t=\beta_t \tll^t \tuu^t\, , \quad \partial_t \tln^t=\beta_t \tln^t\tuu^t\, , \quad \partial_t \tnn^t=\beta_t \tnn^t\tuu^t\, .
\end{equation*}

\noindent
whose general solution is given in the statement of the lemma.
\end{proof}
 
\begin{remark}
\label{rem:maxintquasi}
Let $\Theta_{uu}\neq 0$ and define $t_0$ to be the real number (in case it exists) with the smallest absolute value such that:
\begin{equation*}
\int_{0}^{t_0} \beta_\tau \dd \tau = \Theta_{uu}^{-1}\, .
\end{equation*}

\noindent
Then the maximal interval on which $\Theta^t$ is defined is $\cI= (-\infty, t_0)$ if $\Theta_{uu} > 0$ and $\cI= (t_0,\infty)$ if $\Theta_{uu} < 0$. This is also the maximal interval on which the left-invariant parallel spinor flow in the quasi-diagonal case can be defined. If such $t_0$ does not exist, then $\cI=\mathbb{R}$.
\end{remark}

\noindent
We consider now the non-quasi-diagonal case $\lambda  \neq 0$. Given a pair $\left\{\beta_t,\Theta^t_{ab} \right\}_{t\in \cI}$, we introduce for convenience the following function: 
\begin{equation*}
\cI \ni t\mapsto y_t = \lambda \, \cB_{t} + \mathrm{Arctan} \left[ \frac{\tuu}{\lambda}\right]\, ,
\end{equation*}

\noindent
where $\Theta_{ab}$ are the components of $\Theta$ in the basis $\fre$.

\begin{lemma}
\label{lemma:Thetageneral}
A pair $\left\{\beta_t,\Theta^t_{ab} \right\}_{t\in \cI}$ satisfies the integrability equations with non-quasi-diagonal initial value $\Theta_{ab}$ if and only if:
\begin{eqnarray*}
& \tuu^t = \lambda\,\mathrm{Tan}\left[y_t \right]\, , \quad \tul^t = \tul\, , \quad \tun^t = \tun\, , \quad \tll^t = c_{ll}\, \mathrm{Sec} \left[ y_t \right] - \frac{\tul^2}{\lambda} \mathrm{Tan}\left[ y_t\right]\, ,\\
& \tnn^t = c_{nn}\, \mathrm{Sec} \left[y_t\right] - \frac{\tun^2}{\lambda} \mathrm{Tan}\left[y_t\right]\, , \quad \tln^t = c_{ln}\, \mathrm{Sec} \left[y_t \right] - \frac{\tul \,\tun}{\lambda} \mathrm{Tan}\left[y_t\right]\, ,
\end{eqnarray*}

\noindent
where $c_{ll}, c_{nn}, c_{ln}\in\mathbb{R}$ are real constants given by:
\begin{equation*}
c_{ll} =  \frac{\tll \lambda^2 + \tul^2 \tuu}{\lambda\,\sqrt{\lambda^2 +   \tuu^2}}\, , \quad c_{nn} =   \frac{\tnn \lambda^2 + \tun^2 \tuu}{\lambda\,\sqrt{\lambda^2 +   \tuu^2}}\, , \quad c_{ln} =  \frac{\tln \lambda^2 + \tul \tun \tuu}{\lambda\,\sqrt{\lambda^2 +   \tuu^2}}\, ,
\end{equation*}

\noindent
such that the following algebraic equations are satisfied:
\begin{eqnarray}
\label{eq:algebraicTheta}
& \tln \tul = \tll \tun\, , \quad  \tnn \tul = \tln \tun\, , \nonumber \\
&\tln \tun + \tul (\tll + \tuu) = 0\, , \quad \tln \tul + \tun (\tnn + \tuu) = 0\, ,
\end{eqnarray}

\noindent
where $\Theta_{ab}$, $a,b = u,l,n$, denote the entries of $\Theta^t_{ab}$ at $t=0$.
\end{lemma}

\begin{remark}
Note that equations \eqref{eq:algebraicTheta} form an algebraic system for the entries of the initial condition $\Theta$, therefore restricting the allowed initial data that can be used to solve the integrability conditions. This is a manifestation of the fact that the initial data of the parallel spinor flow is constrained by the parallel Cauchy equations. The latter were solved in the left-invariant case in \cite{Murcia:2020zig}, as summarized in Theorem \ref{thm:allcauchygroups}, and its solutions can be easily verified to satisfy equations \eqref{eq:algebraicTheta} automatically.
\end{remark}

\begin{proof}
By Lemma \ref{lemma:intecondli} we have $\partial_t\Theta^t_{ul} = \partial_t\Theta^t_{un} = 0$ whence $\tul^t = \tul, \tun^t = \tun$ for some real constants $\tul, \tun \in \mathbb{R}$. Plugging these constants into the first equation of Lemma \ref{lemma:intecondli} it becomes immediately integrable with solution:
\begin{equation*}
\Theta^t_{uu} =   \lambda \mathrm{Tan}\left[ \lambda (\cB_t + k_1)\right]\, ,
\end{equation*} 

\noindent
for a certain constant $k_1\in \mathbb{R}$. Imposing $\Theta_{uu}^0 = \Theta_{uu}$ we obtain:
\begin{equation*}
k_1 = \frac{1}{\lambda}(\mathrm{Arctan} \left[ \frac{\tuu}{\lambda}\right] + n\pi)\, , \qquad n\in \mathbb{Z}\, ,
\end{equation*}

\noindent
and the expression for $\Theta^t_{uu}$ follows. Plugging now $\tuu^t = \lambda\,\mathrm{Tan}\left[y_t \right]$ in the remaining differential equations of Lemma \ref{lemma:intecondli} they can be directly integrated, yielding the expressions in the statement after imposing $\Theta^0_{ab} = \Theta_{ab}$. Plugging the explicit expressions for $\Theta^t_{ab}$ in the algebraic equations of Lemma \ref{lemma:intecondli}, these can be equivalently reformulated as the algebraic system \eqref{eq:algebraicTheta} for $\Theta^t_{ab}$ at $t=0$ and we conclude. 
\end{proof}

\begin{remark}
\label{rem:maxintnonquasi}
Let $t_-<0$ denote the largest value for which $\lambda \cB_{t_-}+ \mathrm{Arctan}\left[ \frac{\tuu}{\lambda} \right]=-\frac{\pi}{2}$ and let $t_+>0$ denote the smallest value for which $\lambda \cB_{t_+}+ \mathrm{Arctan}\left[ \frac{\tuu}{\lambda} \right]=\frac{\pi}{2}$ (if $t_-$, $t_+$ or both do not exist, we take by convention $t_{\pm} =\pm \infty$). Then, the maximal interval of definition on which $\Theta^t$ is defined is $\mathcal{I}=(t_-,t_+)$. 
\end{remark}
  
\subsection{Classification of left-invariant spinor flows}


Proposition \ref{lemma:intecondli} states that $\tul^t=\tul$ and $\tun^t=\tun$ for constants $\tul, \tun \in \mathbb{R}$. Therefore, we proceed to classify left-invariant parallel spinor flows in terms of the possible values of $\tul$ and $\tun$. We begin with the classification of quasi-diagonal left-invariant parallel spinor flows, defined by the condition $\tul=\tun=0$, that is, $\lambda = 0$.  
 
\begin{proposition}
\label{prop:liqd} 
Let $\{\beta_t, \fre^t\}_{t \in \cI}$ be a quasi-diagonal left-invariant parallel spinor flow with initial data $(\fre,\Theta)$ satisfying $\Theta_{uu} \neq 0$. Define $Q$ to be the orthogonal two by two matrix diagonalizing $\theta/\Theta_{uu}$ as follows:
\begin{equation*}
\frac{\theta}{\Theta_{uu}} =   Q	\begin{pmatrix}
		\rho_{+} & 0 \\
		0 & \rho_{-}
	\end{pmatrix}  
 Q^{\ast}
\end{equation*}

\noindent
with eigenvalues $\rho_{+}$ and $\rho_{-}$ and where $Q^\ast$ denotes the matrix transpose of $Q$. Then:
\begin{equation*}
e^t_u = (1-\tuu \cB_t) e_u\, , \quad 	\begin{pmatrix}
		e^t_l \\
		e^t_n  
	\end{pmatrix} = Q 
	\begin{pmatrix}
		\left[1- \Theta_{uu}\cB_t \right]^{\rho_{+}} & 0 \\
		0 &  \left[1- \Theta_{uu}\cB_t \right]^{\rho_{-}}
	\end{pmatrix}  Q^{\ast} \begin{pmatrix}
	e_l \\
	e_n  
\end{pmatrix}
\end{equation*}

\noindent
Conversely, for every family of functions $\left\{ \beta_{t}\right\}_{t\in\cI}$ the previous expression defines a parallel spinor flow on $\G$. The case $\tuu=0$ is recovered by taking the formal limit $\tuu \rightarrow 0$. 
\end{proposition}

\begin{proof}
Define the function $\cI\ni t\to x_t := \mathrm{Log} \left[1-\tuu \cB_t \right]$. By Proposition \ref{prop:bijectionsolutions} and Corollary \ref{cor:solutionsiff} it suffices to use the explicit expression for $\Theta^t$ obtained in Lemma \ref{lemma:Thetaqdiagonal} to solve Equation \eqref{eq:ptu} with initial condition $\U^0 = \mathrm{Id}$ on a simply connected Lie group admitting quasi-diagonal parallel Cauchy pairs. Plugging the explicit expression of $\Theta^t$ in \eqref{eq:ptu} we obtain:
\begin{equation}
\label{eq:ptu2}
\begin{pmatrix}
\partial_t \U_{uu}^t & \partial_t \U_{ul}^t & \partial_t \U_{un}^t \\
\partial_t \U_{lu}^t & \partial_t \U_{ll}^t & \partial_t \U_{ln}^t \\
\partial_t \U_{nu}^t & \partial_t \U_{nl}^t & \partial_t \U_{nn}^t 
\end{pmatrix} = \frac{\partial_t x_t}{\tuu} 
\begin{pmatrix}
\U_{uu}^t \Theta_{uu} & \U_{ul}^t \Theta_{uu} & \U_{un}^t \Theta_{uu}\\
\U_{lu}^t \Theta_{ll} + \U_{nu}^t \Theta_{ln} & \U_{ll}^t \Theta_{ll} + \U_{nl}^t \Theta_{ln} & \U_{ln}^t \Theta_{ll} + \U_{nn}^t \Theta_{ln} \\
\U_{lu}^t \Theta_{ln} + \U_{nu}^t \Theta_{nn} & \U_{ll}^t \Theta_{ln} + \U_{nl}^t \Theta_{nn} & \U_{ln}^t \Theta_{ln} + \U_{nn}^t \Theta_{nn}
\end{pmatrix}  
\end{equation}

\noindent
in terms of the initial data  $\Theta_{ab}$. The previous differential system can be equivalently written as follows:
\begin{equation*}
\partial_t\, \U^t_{uc} = \partial_t x_t \, \U^t_{uc}\, ,  \qquad \begin{pmatrix}
	\partial_t \U_{lc}^t \\
	\partial_t \U_{nc}^t   
\end{pmatrix} = \frac{\partial_t x_t}{\tuu} 
\begin{pmatrix}
\Theta_{ll}   &   \Theta_{ln} \\
\Theta_{ln}   &  \Theta_{nn}  
\end{pmatrix}  \begin{pmatrix}
 \U_{lc}^t \\
 \U_{nc}^t   
\end{pmatrix} \, , \quad c=u,l,n\, .
\end{equation*}

\noindent
The general solution to the equations for $\U^t_{uc}$ with initial condition $\U^0 = \mathrm{Id}$ is given by:
\begin{equation*}
\U^t_{uu}=1-\tuu \cB_t\, , \quad \U^t_{ul}=\U^t_{un} =0\, .
\end{equation*} 

\noindent
Consider now the diagonalization of the constant matrix occurring in the differential equations for $\U^t_{ic}$:
\begin{equation*}
\frac{1}{\tuu} 
\begin{pmatrix}
\Theta_{ll}   &   \Theta_{ln} \\
\Theta_{ln}   &  \Theta_{nn}  
\end{pmatrix}   = Q\begin{pmatrix}
\rho_{+}   &   0 \\
0   &  \rho_{-}  
\end{pmatrix} Q^{\ast} \, ,
\end{equation*}

\noindent
where $Q$ is a two by two orthogonal matrix and $Q^{\ast}$ is its transpose. The eigenvalues are explicitly given by:
\begin{equation*}
\rho_{\pm} = \frac{T \pm \sqrt{T^2- 4\Delta}}{2\Theta_{uu}} \, .  
\end{equation*}

\noindent
We obtain:
\begin{equation*}
Q^{\ast} \begin{pmatrix}
		\partial_t \U_{lc}^t \\
		\partial_t \U_{nc}^t   
	\end{pmatrix} = \partial_t x_t 
	\begin{pmatrix}
		\rho_{+}   &   0 \\
		0   &  \rho_{-}  
	\end{pmatrix} Q^{\ast}  \begin{pmatrix}
		\U_{lc}^t \\
		\U_{nc}^t   
	\end{pmatrix} \, , \quad c=u,l,n\, , 
\end{equation*}

\noindent
whose general solution is given by:
\begin{equation*}
\begin{pmatrix}
\U_{lc}^t \\
\U_{nc}^t   
\end{pmatrix} = Q \begin{pmatrix}
k^{+}_c e^{\rho_{+} x_t} \\
k^{-}_c e^{\rho_{-} x_t}   
\end{pmatrix} = Q \begin{pmatrix}
k^{+}_c \left[1- \Theta_{uu}\cB_t \right]^{\rho_{+}} \\
k^{-}_c \left[1- \Theta_{uu}\cB_t \right]^{\rho_{-}}  
\end{pmatrix} \, , \quad c=u,l,n\, , 
\end{equation*}

\noindent
for constants $k^{+}_c, k^{-}_c\in \mathbb{R}$. Imposing the initial condition $\U^0 =\mathrm{Id}$ we obtain the following expression for $k^{+}_c$ and $k^{-}_c$:
\begin{equation*}
\begin{pmatrix}
		k^{+}_c   \\
		k^{-}_c     
	\end{pmatrix} = Q^{\ast}\begin{pmatrix}
	\delta_{lc} \\
	\delta_{nc}   
\end{pmatrix}  \, , \quad c=u,l,n\, , 
\end{equation*}

\noindent
whence:
\begin{equation*}
	\begin{pmatrix}
		k^{+}_u   \\
		k^{-}_u     
	\end{pmatrix} =0 \, , \quad \begin{pmatrix}
	k^{+}_l   \\
	k^{-}_l     
\end{pmatrix} = Q^{\ast}\begin{pmatrix}
1 \\
0    
\end{pmatrix} = \begin{pmatrix}
Q^{\ast}_{ll} \\
Q^{\ast}_{nl}  
\end{pmatrix} \, , \quad \begin{pmatrix}
k^{+}_n   \\
k^{-}_n     
\end{pmatrix} = Q^{\ast}\begin{pmatrix}
0 \\
1   
\end{pmatrix} = \begin{pmatrix}
Q^{\ast}_{ln} \\
Q^{\ast}_{nn}  
\end{pmatrix}  \, . 
\end{equation*}

\noindent
We conclude that:
\begin{equation*}
\begin{pmatrix}
\U_{ll}^t &  \U_{ln}^t \\
\U_{nl}^t &   \U_{nn}^t 
\end{pmatrix} =  Q 
\begin{pmatrix}
e^{\rho_{+} x_t} & 0 \\
0 &  e^{\rho_{-} x_t}
\end{pmatrix}  Q^{\ast} = Q 
\begin{pmatrix}
\left[1- \Theta_{uu}\cB_t \right]^{\rho_{+}} & 0 \\
0 &  \left[1- \Theta_{uu}\cB_t \right]^{\rho_{-}}
\end{pmatrix}  Q^{\ast}
\end{equation*}

\noindent
and the statement is proven. The converse follows by construction upon use of Lemma \ref{lemma:cKTheta} and Proposition \ref{prop:bijectionsolutions}. It can be easily seen that the case $\Theta_{uu} =0$ is obtained by taking the formal limit $\tuu \rightarrow 0$ and we conclude.
\end{proof}

\begin{remark}
The Ricci tensor of the family of Riemannian metrics $\left\{ h_{\fre^t}\right\}_{t\in \cI}$ associated to a left-invariant quasi-diagonal parallel spinor flow $\left\{\beta_t,\fre^t\right\}_{t\in\cI}$ is given by:
\begin{equation}
\ric^{h_{\fre^t}}=-T^t\Theta_t+\frac{\mathcal{H}_t}{2} e_u^t \otimes e_u^t\, ,
\end{equation}

\noindent
where $T^t = \Theta^t_{ll} + \Theta^t_{nn}$. If $\cH_t = 0$ for every $t\in \cI$, that is, if the parallel Cauchy pair defined by $\left\{\beta_t,\fre^t\right\}_{t\in\cI}$ is constrained Ricci flat, then:
\begin{equation}
\ric^{h_{\fre^t}}=\frac{T^t}{2\beta_t}\partial_t h_{\fre^t}\, ,
\end{equation}

\noindent
which, after a reparametrization of the time coordinate can be brought into the form $\ric^{h_{\fre^{\tau}}} = -2\partial_{\tau} h_{\fre^{\tau}}$ after possibly shrinking $\cI$. Hence, this gives a particular example of a left-invariant Ricci flow on $\G$.
\end{remark}

\noindent
We consider now $\tul \tun=0$ but $\tul^2+\tun^2 \neq 0$. This case necessarily corresponds to $\G = \tau_2 \oplus \mathbb{R}$.

\begin{proposition}
\label{prop:linocero}
Let $\{\beta_t, \fre^t\}_{t \in \cI}$ be a left-invariant parallel spinor flow  with initial parallel Cauchy pair $(\fre,\Theta)$ satisfying $\tul \tun=0$ and $\lambda \neq 0$. Then:
\begin{itemize}[leftmargin=*]
\item If $\tul=0 $ the following holds:
\begin{eqnarray*}
& e_u^t =\left (1-\tuu \cB_t \right )\, e_u- \tun\cB_t\, \, e_n\, , \quad e_l^t = e_l \, , \\
& e_n^t = \left ( \frac{\tuu}{\tun}-\frac{\lambda}{\tun}(1-\tuu\cB_t  )\mathrm{Tan}\left[y_t \right]\right) e_u +\left (1+\lambda \cB_t \, \mathrm{Tan}\left[y_t\right]\right) e_n\, .
\end{eqnarray*}

\item  If $\tun=0 $ the following holds:
\begin{eqnarray*}
& e_u^t =\left (1-\tuu \cB_t \right )\, e_u- \tul \cB_t\, \, e_l\, , \quad e_n^t = e_n \, , \\
& e_l^t = \left ( \frac{\tuu}{\tul}-\frac{\lambda}{\tul}(1-\tuu\cB_t  )\mathrm{Tan}\left[y_t \right]\right) e_u +\left (1+\lambda \cB_t \, \mathrm{Tan}\left[y_t\right]\right) e_l\, .
\end{eqnarray*}
\end{itemize}

\noindent
Conversely, every such family $\{\beta_t, \fre^t\}_{t \in \cI}$ is a left-invariant parallel spinor flow for every $\{\beta_t\}_{t \in \cI}$. 
\end{proposition}

\begin{proof}
We prove the case $\tul=0$ and $\tun \neq 0$ since the case $\tun=0$ and $\tul \neq 0$ follows similarly. Setting $\tul=0$ and assuming $\tun \neq 0$ in Lemma \ref{lemma:Thetageneral} we immediately obtain:
\begin{equation*}
\tuu^t=-\tnn^t=\lambda\mathrm{Tan} \left[ y_t\right]\, , \qquad \Theta^t_{ll} = \Theta^t_{ln} = 0\, .
\end{equation*}

\noindent
where we have also used that, in this case, $\Theta_{ln} = \Theta_{ll} = 0$ and $\tuu=-\tnn$ as summarized in Theorem \ref{thm:allcauchygroups}. Hence: 
\begin{equation*}
\Theta^t = \Theta_{un}
\begin{pmatrix}
0 & 0 & 1\\
0 & 0 & 0\\
1 & 0 & 0
\end{pmatrix}  +  \lambda\mathrm{Tan} \left[ y_t\right]
\begin{pmatrix}
1 & 0 & 0\\
0 & 0 & 0\\
0 & 0 & -1
\end{pmatrix} \, ,
\end{equation*}

\noindent
and Equation \eqref{eq:ptu} reduces to:
\begin{equation*}
\partial_t \U^t + \beta_t \Theta_{un} \begin{pmatrix}
	\U_{nu}^t & \U_{nl}^t & \U_{nn}^t\\
	0 & 0 & 0\\
	\U_{uu}^t & \U_{ul}^t & \U_{un}^t 
\end{pmatrix} + \lambda \beta_t \mathrm{Tan} \left[ y_t\right]
\begin{pmatrix}
\U_{uu}^t & \U_{ul}^t & \U_{un}^t\\
0 & 0 & 0\\
-\U_{nu}^t & -\U_{nl}^t & -\U_{nn}^t
\end{pmatrix} = 0 \, ,
\end{equation*}

\noindent
or, equivalently:
\begin{eqnarray*}
\partial_t \U^t_{uc} + \beta_t \lambda\mathrm{Tan} \left[ y_t\right] \U^t_{uc}+\beta_t \tun \U^t_{nc} = 0\, , \quad \partial_t\, \U^t_{lc} = 0\, , \quad
\partial_t\, \U^t_{nc} -\beta_t \lambda\mathrm{Tan} \left[ y_t\right] \U_{nc}+\beta_t \tun \U^t_{uc} =0\, . 
\end{eqnarray*}

\noindent
The general solution to this system with initial condition $\U^{0}=\mathrm{Id}$ is given by:
\begin{eqnarray*}
& \U^t_{uu} = 1-\tuu \cB_t\, , \quad \U^t_{un}=-\tun \cB_t\, , \quad \U^t_{ul}=\U^t_{lu}=\U^t_{ln}=\U^t_{nl}=0\, , \quad \U^t_{ll}=1\,, \\ 
& \U^t_{nn} = 1+ \lambda \cB_t  \mathrm{Tan}\left[ y_t \right]\, , \quad \U^t_{nu} = \frac{\tuu}{\tun}-\frac{\lambda}{\tun}(1-\tuu\cB_t  ) \mathrm{Tan}\left[ y_t \right] \, ,
\end{eqnarray*}

\noindent
which implies the statement. The converse follows by construction upon use of Lemma \ref{lemma:cKTheta} and Proposition \ref{prop:bijectionsolutions}.
\end{proof}

\begin{remark}
The Ricci tensor of the family of metrics $\left\{ h_{\fre^t}\right\}_{t\in\cI}$ associated to a left-invariant parallel spinor with  if $\tul=0$ but $\tun \neq 0$ reads:
\begin{equation*}
\ric^{h_{\fre^t}}=-\Theta_t \circ \Theta_t=\frac{\mathcal{H}_t}{4}  (h_{\fre^t}-e_n^t \otimes e_n^t)\, .
\end{equation*}

\noindent
Recall that $\nabla^{h_{\fre^t}} e_n^t=0$ and thus $\left\{h_{\fre^t}, e_n^t \right\}_{t\in\cI}$ defines a family of $\eta\,$-Einstein cosymplectic structures \cite{Capelletti,Olszak}. On the other hand, if  $\tun=0$ but $\tul \neq 0$ the curvature of $\left\{ h_{\fre^t}\right\}_{t\in\cI}$ is given by: 
\begin{equation*}
\ric^{h_{\fre^t}}=-\Theta_t \circ \Theta_t=\frac{\mathcal{H}_t}{4}  (h_{\fre^t}-e_l^t \otimes e_l^t)\, .
\end{equation*}

\noindent
whence $\{ h_{\fre^t}, e_l^t\}_{t\in\cI}$ defines as well a family of $\eta\,$-Einstein cosymplectic structures on $\G$.
\end{remark}

\noindent
Finally we consider $\tul \tun \neq 0$, a case that again corresponds to $\G = \tau_{2} \oplus \mathbb{R}$.
\begin{proposition}
\label{prop:tultunneq0}
Let $\{\beta_t, \fre^t\}_{t \in \cI}$ be a left-invariant parallel spinor flow  with initial parallel Cauchy pair $(\fre,\Theta)$ satisfying $\tul \tun \neq 0$. Then:
\begin{eqnarray*}
& e_u^t = e_u+\Bt (T e_u-\tul e_l-\tun e_n)\, ,\\
& e_l^t = -\frac{\tul}{\lambda} \left( \frac{T}{\lambda} + \left (1+ T \Bt\right ) \mathrm{Tan} [y_t] \right)\, e_u +\left ( 1+\frac{\tul^2 \mathcal{B}_t}{\lambda} \mathrm{Tan} [y_t]  \right) e_l+\frac{\tul \tun \Bt }{\lambda} \mathrm{Tan} [y_t] \, \, e_n\, ,\\ 
& e_n^t = -\frac{\tun}{\lambda} \left (\frac{T}{\lambda} + \left ( 1+ T \Bt \right) \mathrm{Tan} [y_t]\right) \, e_u  + \frac{\tul \tun \Bt}{\lambda}  \mathrm{Tan} [y_t] \, e_l+\left (1+\frac{\tun^2 \Bt}{ \lambda} \ \mathrm{Tan} [y_t]  \right )\, e_n\, ,  
\end{eqnarray*}

\noindent
Conversely, every such family $\{\beta_t, \fre^t\}_{t \in \cI}$ is a left-invariant parallel spinor flow for every $\{\beta_t\}_{t \in \cI}$. 
\end{proposition}

\begin{proof}
Assuming $\tul,\tun\neq 0$ in Lemma \ref{lemma:Thetageneral} we obtain:
\begin{eqnarray*}
\tuu^t = \lambda \mathrm{Tan}\left[ y_t \right]\, , \quad \tll^t =\frac{\tul}{\tun}\tln^t\, , \quad \tnn^t=\frac{\tun}{\tul} \tln^t \, , \quad \tln^t=-\frac{\tul \tun}{\tul^2+\tun^2} \tuu^t\, .
\end{eqnarray*}

\noindent
Note that $\tuu^t=-\tll^t-\tnn^t$. Hence:
\begin{equation*}
\Theta^t = 
\begin{pmatrix}
0 & \Theta_{ul} & \Theta_{un}\\
\Theta_{ul} & 0 & 0\\
\Theta_{un} & 0 & 0
\end{pmatrix}  -  \frac{\mathrm{Tan} \left[ y_t\right]}{\lambda}
\begin{pmatrix}
-\lambda^2  & 0 & 0\\
0 & \Theta_{ul}^2 & \Theta_{ul}\Theta_{un}\\
0 & \Theta_{ul}\Theta_{un} & \Theta_{un}^2
\end{pmatrix} \, ,
\end{equation*}

\noindent
and Equation \eqref{eq:ptu} reduces to:
\begin{eqnarray*}
& \frac{1}{\beta_t}\partial_t \U^t_{uc} + \U^t_{lc} \tul + \U^t_{nc} \tun + \U^t_{uc} \tuu^t  =0\, ,\\
& \frac{1}{\beta_t}\partial_t \U^t_{lc} + \tul \left (\U^t_{uc}- \lambda^{-1} (\tul \U^t_{lc}+\tun \U_{nc}^t) \mathrm{Tan} \left[ y_t\right] \right ) =0\, ,\\
& \frac{1}{\beta_t}\partial_t \U^t_{nc} + \tun \left (\U^t_{uc}- \lambda^{-1} (\tul \U^t_{lc}+\tun \U_{nc}^t) \mathrm{Tan} \left[ y_t\right] \right ) =0\, .
\end{eqnarray*}

\noindent
The general solution to this system with initial condition $\U^{0}=\mathrm{Id}$ is given by:
\begin{eqnarray*}
& \U^t_{uu} =1 - \Theta_{uu} \Bt\, , \quad \U^t_{ul}=-\tul \Bt\, , \quad \U^t_{un}=-\tun \Bt\, ,\\ 
& \U^t_{lu} =  \frac{\tul}{\lambda} \left(\frac{\Theta_{uu}}{\lambda} - \left (1 - \Theta_{uu}\Bt\right ) \mathrm{Tan}[y_t] \right)\, ,\quad \U^t_{ll} = 1+\frac{\tul^2 \mathcal{B}_t}{\lambda} \mathrm{Tan}[y_t] \, , \quad \U^t_{ln}=\frac{\tul \tun \Bt }{\lambda} \mathrm{Tan}[y_t]\, , \\ 
& \U^t_{nu} =\frac{\tun}{\lambda} \left(\frac{\Theta_{uu}}{\lambda} - \left (1 - \Theta_{uu}\Bt\right ) \mathrm{Tan}[y_t] \right)\, , \quad  \U^t_{nl} =\frac{\tul \tun \Bt}{\lambda} \mathrm{Tan}[y_t]\, , \quad 
\U^t_{nn}=1+\frac{\tun^2 \Bt}{\lambda}  \mathrm{Tan}[y_t] \, , 
\end{eqnarray*}

\noindent
and we conclude.
\end{proof}
 
\begin{remark}
The three-dimensional Ricci tensor of the family of Riemannian metrics $\left\{h_{\fre^t}\right\}_{t\in \cI}$ associated to a left invariant parallel spinor flow with $\Theta_{ul}\Theta_{un} \neq 0$ reads:
\begin{equation*}
\ric^{h_{\fre^t}}=-\Theta_t \circ \Theta_t =\frac{\mathcal{H}_t}{4} (h_{\fre^t}-\eta_t \otimes \eta_t)\, , \quad \eta_t=\frac{1}{\sqrt{\tul^2+\tun^2}}( \tun e_l^t-\tul e_n^t)\, ,
\end{equation*}

\noindent
Note that $\nabla^{h_{\fre^t}} \eta_t=0$, so $\{ h_{\fre^t},\eta_t\}_{t \in \cI}$ defines a family of $\eta\,$-Einstein cosymplectic Riemannian structures on $\G$.
\end{remark}
 
\noindent
As a corollary to the classification of left-invariant parallel spinor flows presented in Propositions \ref{prop:liqd}, \ref{prop:linocero} and \ref{prop:tultunneq0} we can explicitly obtain the evolution of the Hamiltonian constraint in each case. 

\begin{corollary}
Let $\left\{\beta_t,\fre^t\right\}_{t\in\cI}$ be a left-invariant parallel spinor in $(M,g)$. 
\begin{itemize}[leftmargin=*]
\item If $\tul=\tun=0$, then $\mathcal{H}_t=\frac{\mathcal{H}_{0}}{(1-\tuu \cB_t)^2}$.
 
\item If $\tul=0$ but $\tun \neq 0$ then $\mathcal{H}_t=\frac{\tun^2 \mathcal{H}_{0}}{\tuu^2+\tun^2} \sec^2 \left[    \lambda  \cB_t  + \mathrm{Arctan}\left[ \frac{\tuu}{\lambda} \right] \right]$.

\item If $\tun=0$ but $\tul \neq 0$ then $\mathcal{H}_t=\frac{\tul^2 \mathcal{H}_{0}}{\tuu^2+\tul^2} \sec^2 \left[     \lambda \cB_t  + \mathrm{Arctan}\left[ \frac{\tuu}{\lambda}\right] \right]$.

\item If $\tul,\tun \neq 0$ then $\mathcal{H}_t=\frac{\lambda^2 \mathcal{H}_{0}}{\lambda^2 + \tuu^2} \sec^2 \left[ \lambda \cB_t + \mathrm{Arctan}\left[ \frac{\tuu}{\lambda}\right] \right] $.
\end{itemize}

\noindent
where $\cH_0$ is the Hamiltonian constraint at time $t=0$. 
\end{corollary}

\noindent
Since the secant function has no zeroes, the Hamiltonian constraint vanishes for a given $t\in \cI$, and hence for every $t\in \cI$, if and only if it vanish at $t=0$, consistently with Theorem \ref{thm:preservingconstraints}. Theorem \ref{thm:allcauchygroups} implies that only quasi-diagonal left-invariant parallel spinor flows admit constrained Ricci flat initial data. Therefore the Hamiltonian constraint of left-invariant parallel spinor flows with $\lambda \neq 0$ is non-vanishing for every $t\in\cI$ and such left-invariant parallel spinor flows cannot produce four-dimensional Ricci flat Lorentzian metrics. 


\subsection{Proof of Theorem \ref{thm:hflows}}


Theorem \ref{thm:hflows} follows through a direct computation by using the explicit form of the left-invariant parallel spinor flow obtained in Propositions \ref{prop:liqd}, \ref{prop:linocero} and \ref{prop:tultunneq0} for each of the possible cases, after using Theorem \ref{thm:allcauchygroups} to identify the underlying Lie group in each case.


\appendix


\phantomsection
\bibliographystyle{JHEP}


\end{document}